%% file: main.tex
\pdfoutput=1
\documentclass[11pt,a4paper]{amsart}

\usepackage{comment}
\usepackage{ifluatex}
\usepackage[english]{babel} 

\usepackage{amscd,amsmath,amssymb,amsthm} 
\usepackage{multirow}                     

\ifluatex
  \usepackage{fontspec}
\else
  \usepackage[utf8]{inputenc} 
  \usepackage{lmodern}
  \usepackage[T1]{fontenc}
\fi

\usepackage[margin=1in]{geometry}

\usepackage{mathtools}
\usepackage{bm} 

\usepackage{slashed} 
\usepackage[babel=true]{microtype}
\usepackage[autostyle=true]{csquotes} 
\usepackage[dvipsnames]{xcolor}
\usepackage[biblatex=true]{embrac}

\usepackage[%
bookmarks=true,
colorlinks,
linkcolor=blue,
urlcolor=blue,
citecolor=blue,
plainpages=false,
pdfpagelabels,
final,
breaklinks=true,
pdfusetitle,
]{hyperref}
\usepackage[citestyle=numeric-comp,
            bibstyle=numeric-comp,
            backend=biber,
            url=false,
            doi=true,
            isbn=false,
            giveninits=true,
            maxbibnames=100,
            sorting=nyt]{biblatex}

\AtEveryBibitem{\clearfield{month}}
\AtEveryBibitem{\clearfield{day}}
\AtEveryBibitem{%
  \ifentrytype{thesis}
    {}
    {
      \ifentrytype{online}{}
      {
      \clearfield{url}%
      \clearfield{urldate}%
      }
    }%
}
            
\usepackage{tikz}
\usetikzlibrary{cd} 

\usepackage[all]{xy}

\usepackage[shortlabels]{enumitem}
\newlist{myenumi}{enumerate}{1}
\setlist[myenumi,1]{label=\upshape(\roman*)}
\newlist{myenuma}{enumerate}{1}
\setlist[myenuma,1]{label=\upshape(\alph*)}
\usepackage{color}
\usepackage[textwidth=0.8in]{todonotes}
\usepackage{thmtools}
\declaretheorem[name=Theorem, numberwithin=section]{theorem}
\declaretheorem[name=Theorem, numbered=no]{theorem*}
\declaretheorem[name=Lemma,numberlike=theorem]{lemma}
\declaretheorem[name=Lemma,numbered=no]{lemma*}
\declaretheorem[name=Corollary,numberlike=theorem]{cor}
\declaretheorem[name=Proposition,numberlike=theorem]{prop}
\declaretheorem[name=Definition,numberlike=theorem, style=definition]{definition}

\declaretheorem[name=Conditions, numberlike=theorem, style=definition]{Conditions}

\declaretheorem[name=Example, numberlike=theorem, style=remark]{example}
\declaretheorem[name=Remark, numberlike=theorem, style=remark]{rem}

\declaretheorem[name=Theorem]{thmx}
\declaretheorem[name=Corollary, numberlike=thmx]{corx}

\numberwithin{equation}{section}
\allowdisplaybreaks[1]
\usepackage[noabbrev]{cleveref} 
\crefname{figure}{Figure}{Figures}
\crefname{table}{Table}{Tables}
\crefname{theorem}{Theorem}{Theorems}
\crefname{thmx}{Theorem}{Theorems}
\crefname{lemma}{Lemma}{Lemmas}
\crefname{definition}{Definition}{Definitions}
\crefname{setup}{Setup}{Setups}
\crefname{conjecture}{Conjecture}{Conjectures}
\crefname{question}{Question}{Questions}
\crefname{cor}{Corollary}{Corollaries}
\crefname{corx}{Corollary}{Corollaries}
\crefname{prop}{Proposition}{Propositions}
\crefname{example}{Example}{Examples}
\crefname{rem}{Remark}{Remarks}
\crefname{section}{Section}{Sections}
\crefname{subsection}{Subsection}{Subsections}
\crefname{chapter}{Chapter}{Chapters}
\crefname{appendix}{Appendix}{Appendices}
\crefdefaultlabelformat{#2\textup{#1}#3}
\creflabelformat{enumi}{(#2#1#3)}

\usepackage{crossreftools}
\usepackage{xparse}

\usepackage{ourmacros}

\newcommand{\vphi}{\varphi}

\newcommand{\Hom}{\operatorname{Hom}}
\newcommand{\Hol}{{\operatorname{Hol}}}
\newcommand{\Riem}{\mathrm{Riem}}

\newcommand{\trace}{\mathrm{tr}}
\newcommand{\pb}{\mathsf{pb}}  
\newcommand{\std}{\mathrm{std}} 
\newcommand{\FinMetric}[2]{\langle\!\langle  #1,\, #2\rangle\!\rangle}
\newcommand{\EH}{{EH}}
\newcommand{\gEH}{g_{\EH}} 
\newcommand{\gEHpar}[1]{g_{\EH,{#1}}} 
\newcommand{\vertical}{\mathrm{vert}}

\newcommand{\pr}{\mathrm{pr}}

\newcommand{\GTwostd}{\varphi^{\mathrm{std}}}

\title[Non-aspherical Moduli Spaces]{Path Components of $\GTwo$-Moduli Spaces may be non-aspherical}

\hypersetup{
  pdfauthor={Diarmuid Crowley, Sebastian Goette, and Thorsten Hertl}
}

\author{Diarmuid Crowley}
\address[D.~Crowley]{School of Mathematics and Statistics, The University of Melbourne, Australia}
\email{\href{mailto:dcrowley.unimelb.edu.au}{dcrowley@unimelb.edu.au}}
\urladdr{\href{https://dcrowley.net}{dcrowley.net}}

\author{Sebastian Goette}
\address[S.~Goette]{Mathematisches Institut, Albert-Ludwigs Universit\"at Freiburg, Germany}
\email{\href{mailto:sebastian.goette@math.uni-freiburg.de}{sebastian.goette@math.uni-freiburg.de}}
\urladdr{\href{https://home.mathematik.uni-freiburg.de/goette/}{https://home.mathematik.uni-freiburg.de/goette/}}

\author{Thorsten Hertl}
\address[T.~Hertl]{School of Mathematics and Statistics, The University of Melbourne, Australia }
\email{\href{mailto:thorsten.hertl@unimelb.edu.au}{thorsten.hertl@unimelb.edu.au}}
\urladdr{\href{https://thorsten-hertl.github.io/}{https://thorsten-hertl.github.io/}}

\date{\today}

\begin{document}

\begin{abstract}
  Starting from Joyce's generalised Kummer construction, we exhibit non-trivial families of $\mathrm{G}_2$-manifolds over the two dimensional sphere by resolving singularities with a twisted family of Eguchi-Hanson spaces. 
  We establish that the comparison map $\mathcal{G}_2^{\mathrm{tf}}(M) /\!\!/ \mathrm{Diff}(M)_0 \rightarrow \mathcal{G}_2^{\mathrm{tf}}(M) / \mathrm{Diff}(M)_0$ is a fibration over each path component with Eilenberg Mac Lane spaces as fibres, which allows us to show that these families remain non-trivial in $\mathcal{G}_2^{\mathrm{tf}}(M) / \mathrm{Diff}(M)_0$. 
  In addition, we construct a new invariant based on characteristic classes that allows us to show that different resolutions give rise to different elements in the moduli space.
\end{abstract}

\maketitle


\input{Introduction}

\input{FoundationsModuliSpaces}
\input{ComparisonTheorem}
\input{EguchiHansonBundle}

\input{GeneralisedKummer}

\input{BundleResolution}

\input{ResolutionDetection}

\appendix 
\input{SliceTheorem}



\printbibliography

\end{document}

D.~Crowley is partially supported by the Australian Research Council Discovery Project DP DP220102163.
S.~Goettte is partially funded by the Simons Collaboration "Special Holonomy in Geometry, Analysis and Physics", Grant Number 488617
T.~Hertl acknowledges support by the Simons Collaboration "Special Holonomy in Geometry, Analysis and Physics", Grant Number 488617, and by the Australian Research Council Discovery Project DP DP220102163.

%% file: Introduction.tex
\section{Introduction}\label{Section - Introduction}

In the article \cite{Joyce1996CompactG2I}, Dominic Joyce not only presented the first example of a closed manifold with holonomy group $\GTwo$,
but also provided a structural result regarding the deformation behaviour of $\GTwo$-holonomy metrics on any closed, seven-dimensional manifold.
Joyce showed that the \emph{moduli space}, which he defines as the quotient $\Moduli{M}$ of all parallel or \emph{torsion-free} $\GTwo$-structures by the group of diffeomorphisms isotopic to the identity, is a smooth manifold of dimension $b_3(M)$ by proving that the tautological map $\Moduli{M} \rightarrow H^3(M;\R)$ that sends an equivalence class of a parallel $\GTwo$-structure to its cohomology class is a local diffeomorphism.

This theorem is local in nature, and results about the global topological properties of $\Moduli{M}$ remain scarce.
Indeed, so far, it is only known by work of \cite{crowley2025analytic} and \cite{wallis2018disconnecting} that certain $\GTwo$-manifolds have disconnected moduli spaces; the question whether all path components are contractible has remained open.

The first main result of this article shows that the moduli space of Joyce's first example, the one presented in \cite{Joyce1996CompactG2I}, has a non-contractible path component.

\begin{thmx}\label{Main Theorem - Example}
  If $(M, \varphi_{\mathrm{Joyce}})$ denotes Joyce's first example of a closed $\GTwo$-manifold, then the second homotopy group contains at least $3^{12}$ elements: 
  \begin{equation*}
      \left|\pi_2\left(\Moduli{M},[\varphi_{\mathrm{Joyce}}]\right) \right| \geq 3^{12} = 531\, 441.
  \end{equation*}
\end{thmx}

Besides the challenges of constructing examples of torsion-free $\GTwo$-forms on closed manifolds, a principal obstacle in finding non-trivial elements in homotopy groups of $\GTwo$-moduli spaces stems from its very definition: 
From a categorical perspective, a quotient $\Moduli{M}$ is determined by $\Diff(M)_0$-invariant maps \emph{out} of $\tfGTwoStr(M)$, while elements in the homotopy groups are represented by maps from spheres \emph{into} this quotient, and these maps do not lift to $\tfGTwoStr(M)$ automatically.

To overcome this difficulty, we make use of the \emph{homotopy quotient} $\hModuli{M}$ instead.
The universal property of this topological space relates maps \emph{into} $\hModuli{M}$ to $M$-fibre bundles over the domain with a fibrewise torsion-free $\GTwo$-structure on it, see Section \ref{Section - Comparison} for details.
The homotopy moduli space may be explicitly defined as a quotient $(\tfGTwoStr(M) \times E\Diff(M)_0)/\Diff(M)_0$, where $E\Diff(M)_0$ is a contractible space on which $\Diff(M)_0$ acts freely, and we consider the diagonal action.
The product projections induce two maps
\begin{equation*}
    \xymatrix@R-.5em{ \hModuli{M} \ar[r] \ar[d] & \Moduli{M} \\
     B\Diff(M)_0. &  }
\end{equation*}

A priori, the homotopy moduli space and the actual moduli space might be very different spaces.
However, our second main result, the following comparison theorem, implies that these two spaces carry almost the same information.

\begin{thmx}[Orbit Comparison Theorem]\label{Main Theorem - Comparison Result}
    If $\tfGTwoStr(M)_\varphi$ denotes the path component of $\varphi$ in $\tfGTwoStr(M)$ and if $\mathrm{Stab}_\varphi(\Diff(M)_0)$ denotes the subgroup of elements of~$\Diff(M)_0$ that preserve~$\varphi$, then the canonical comparison map
    \begin{equation*}
        p\colon \tfGTwoStr(M)_\varphi/\!\!/\Diff(M)_0 \rightarrow \tfGTwoStr(M)_\varphi/\Diff(M)_0 
    \end{equation*}
   is a fibration with fibre $B\mathrm{Stab}_{\varphi}(\Diff(M)_0).$
\end{thmx}

Since $\mathrm{Stab}_\varphi(\Diff(M)_0)$ is a finite group if the holonomy group of the induced Riemannian metric $g_\varphi$ is the entire group $\GTwo$, we have the following consequence for the homotopy groups of these spaces.

\begin{corx}\label{Main Corollary - Coconnectivity}
   If $M$ is a manifold that carries a Riemannian metric whose holonomy group is $\GTwo$, then the canonical comparison map $p\colon \tfGTwoStr(M)/\!\!/\Diff(M)_0 \rightarrow \tfGTwoStr(M)/\Diff(M)_0$ is $2$-coconnected, which means its induced homomorphism between the $k$-th homotopy groups is injective if $k=2$ and bijective if $k \geq 3$.
\end{corx}

The important consequence of the Orbit Comparison Theorem and Corollary \ref{Main Corollary - Coconnectivity} is that it allows us to construct non-trivial elements in higher homotopy groups that can be detected using metric-independent methods from differential topology.  
In fact, the Orbit Comparison Theorem is a consequence of general fact about special group actions, see Theorem \ref{Theorem - HQuoVsQuo}, so there are corresponding statements for the moduli space of torsion-free $\Spin(7)$-structures and the moduli space for torsion-free Calabi-Yau structures. 

We expect that the Orbit Comparison Theorem will be useful in constructing elements in the homotopy groups of the moduli spaces of torsion-free $\GTwo$-structures on closed manifolds. 
To our knowledge, all examples of closed $\GTwo$-manifolds were explicitly constructed, using various recipes, and these constructions often involve choices. 
As an illustration, let us sketch how we use Orbit Comparison Theorem to prove Theorem~\ref{Main Theorem - Example}:

We start with the construction of bundles of $\GTwo$-manifolds whose fibre is Joyce's first example $M$ by resolving each singularity of the flat orbifold $T^7/\Gamma$ from which $M$ arises with either a twisted family of Eguchi-Hanson spaces $\mathcal{EH}$ over $S^2$ or with an untwisted family $EH \times S^2$.
Using the first Pontryagin class of the vertical tangent bundle on $E$ and the fact that the singularity set of this flat orbifold has $12$ connected components, we are able to prove the following result.

\begin{thmx}\label{Main Theorem - pi_2-HMS}
The group $\pi_2\left(\hModuli{M},[\varphi_{\mathrm{Joyce}}]\right)$ has at least $531\, 441 = 3^{12}$ elements.    
\end{thmx}



\noindent
The Orbit Comparison Theorem now implies that the image of $\pi_2(\hModuli{M})$ in $\pi_2(\Moduli{M})$ also contains at least $3^{12}$ elements.

\vspace{6pt}

To get a better intuition of what happens geometrically, let us compare our main result with the description of hyperkähler structures on $K3$ surfaces.
Let~$q$ be a nondegenerate, even, unimodular quadratic form
of signature~$(3,19)$ on~$\Lambda\cong\Z^{22} \cong H^2(K3,\Z)$,
then~$(\Lambda,q)$ is unique up to isomorphism.
Let~$\Lambda_\R\cong\Lambda\otimes_\Z\R$, and let~$\langle \placeholder, \placeholder \rangle$
denote the indefinite scalar product induced by~$q$.
Consider
\begin{align*}
  \Kall&=\bigl\{\,(\alpha_1,\alpha_2,\alpha_3)\in\Lambda_\R^3\bigm|
  \langle\alpha_i,\alpha_j\rangle=a\,\delta_{ij}\text{ for some }a>0\,\bigr\}\;,\\
  \Ksing&=\bigcup_{\lambda\in\Lambda,q(\lambda)=-2}
  \bigl\{\,(\alpha_1,\alpha_2,\alpha_3)\in\Lambda_\R^3\bigm|
  \langle\alpha_1,\lambda\rangle=\langle\alpha_2,\lambda\rangle=\langle\alpha_3,\lambda\rangle=0\,\bigr\}\;,
\end{align*}
then the hyperk\"ahler moduli space of a $K3$ surface is diffeomorphic
to~$\Khk\cong\Kall\setminus\Ksing$,
see~\cite[Theorem~7.3.16]{Joyce2000SpecialHolonomy}.
Note that~$\Kall\cong O(3,19)/O(19)\times\R^+$,
and~$\Ksing\cap\Kall$ is the union of subspaces of codimension~$3$
of the form~$O(3,18)/O(18)\times\R^+$.
If we are only interested in the space~$\Kmod$ of Riemannian metrics
induced by the hyperk\"ahler structures,
we divide by the action of~$O(3)$ on $3$-frames in~$\Lambda_\R$,
obtaining the symmetric space~$O(3,19)/(O(3)\times O(19))\times\R^+$
of noncompact type.
Now, the singular set is a union of totally geodesic subspaces
of codimension~$3$.
Hence, each of these subspaces introduces nontrivial elements
of~$\pi_2(\Kmod)$ and~$\pi_2(\Khk)$ in their complement.

Geometrically, the Poincar\'e dual of each cohomology class $\lambda \in \Lambda$ with $q(\lambda)=-2$ can be realised by a unique minimal $\CP^1$ with self-intersection $-2$. 
The volume of this minimal submanifold is given by $\sum_{j=1}^3 \langle \alpha_i,\lambda\rangle^2$, so its volume collapses as one approaches the corresponding component of the singular set.
In the Kummer construction, one resolves the~$16$ singularities of $T^4/\Z_2$ by Eguchi-Hanson spaces.
Each of these contains one copy of $\CP^1$ with self-intersection~$-2$, and these degenerate if the neck-size of the Eguchi-Hanson space approaches zero.

In our examples, the geometric situation is similar.
The non-trivial elements we find in $\pi_2(\Moduli{M})$ locally wind around
``holes'' in~$\Moduli{M}$ of codimension~$3$ where a $T^3$-family
of~$\C P^1$s collapses.
It is also possible to describe nontrivial homology classes~$a\in H_3(M)$ such that~$\langle\varphi,a\rangle \rightarrow 0$ if~$(M,\varphi)$ approaches the singularity.
In~\cite{Dwivedi2023Associative}, Dwivedi, Platt and Walpuski have constructed
associative submanifolds that are contained within a small neighbourhood
of the neck of a single $T^3$-family of Eguchi-Hanson spaces inside~$M$.
Let~$a\in H_3(M)$ denote the homology class they represent.
Then~$\langle\varphi,a\rangle$ describes their volume, which approaches~$0$ as the $\GTwo$-structure becomes singular.
There are similar constructions of coassociative submanifolds and $\GTwo$-instantons that degenerate together with the $\GTwo$-structure
in~\cite{Gutwein2024Coassociative}.

However, the above is insufficient to detect non-trivial elements of $\pi_2(\Moduli{M})$. 
First of all, those associative submanifolds are only guaranteed to exist over a certain subset of the moduli space and, secondly, the tautological map $\Moduli{M} \rightarrow H^3(M)$ is only a local diffeomorphism. 
So, even if one can detect a nontrivial~$S^2$ family inside~$\Moduli{M}$ close to a singularity using associative or coassociative submanifolds, there is no obvious argument that prevents this family from being globally trivial.
We therefore follow a different approach.

\vspace{6pt}

We now digress briefly to discuss
the full quotient $\tfGTwoStr(M)/\Diff(M)$, which is a priori only an orbifold,
and it's homotopy orbit space $\tfGTwoStr(M)/\!\!/\Diff(M)$.
The latter space is of independent interest, since it is the classifying space for smooth fibre bundles with fibre a torsion free $\GTwo$-manifold.
The results of \cite{crowley2025analytic} and \cite{wallis2018disconnecting} 
detect different path components in the full quotient $\tfGTwoStr(M)/\Diff(M)$ and not just the moduli space,
and in that case the natural map 
$\pi_0 \tfGTwoStr(M)/\!\!/\Diff(M) \to \pi_0 \tfGTwoStr(M)/\Diff(M)$ is a bijection.
There are canonical comparison maps between the involved objects that 
form the commutative diagram
\begin{equation*}
    \xymatrix{ \hModuli{M} \ar[rr] \ar[d]_{\mathrm{h}\kappa} && \Moduli{M} \ar[d]^\kappa \\
    \tfGTwoStr(M)/\!\!/\Diff(M) \ar[rr] && \tfGTwoStr(M)/\Diff(M). }
\end{equation*}
The map $\mathrm{h}\kappa$ is a fibration with fibre $B\Gamma$, the classifying space of the \emph{mapping class group} $\Gamma := \Diff(M)/\Diff(M)_0$.
As $\Gamma$ is discrete, all higher homotopy groups of $B\Gamma$ vanish and hence $\pi_2(\mathrm{h}\kappa)$ is an injective homomorphism. 
In particular, Theorem \ref{Main Theorem - pi_2-HMS} remains true for the full homotopy quotient. 

\begin{corx}\label{Corollary - pi_2-FHMS}
The group $\pi_2\left(\tfGTwoStr(M)/\!\!/\Diff(M),[\varphi_{\mathrm{Joyce}}]\right)$ has at least $531\, 441 = 3^{12}$ elements.    
\end{corx}

In general, we expect that neither the map $\kappa$ is a fibration, nor that our Orbit Comparison Theorem remains true for the map $\tfGTwoStr(M)/\!\!/\Diff(M) \rightarrow \tfGTwoStr(M)/\Diff(M)$ because the action of $\Diff(M)$ on $\tfGTwoStr(M)$ may have fixpoints, see Remark \ref{Remark - comparison failure fix points} below for further details.
In particular, we cannot guarantee that the elements in Theorem \ref{Main Theorem - Example} survive under $\pi_2(\kappa)$.
A definitive answer would involve a thorough understanding of the fixpoint set of the pullback action of $\Gamma$ on $\Moduli{M}$, which we do not have at the moment.

\vspace{6pt}

All singularities we consider in this article can be resolved with Eguchi-Hanson spaces and quotients of it.
It is an interesting question whether other crepant resolutions of ADE-singularities may allow us to construct non-trivial elements in homotopy groups beyond degree $2$, and we plan to address this question a in future paper.

Finally, we would like to remark that the non-trivial bundle $\mathcal{EH}$ of Eguchi-Hanson spaces was used by Jiafeng Lin in \cite{lin2022family} in an entirely different context.

\vspace{10pt}
\noindent
\textbf{Outline of the Paper:}
In Section \ref{Section-Foundations}, we recall the necessary background of $\GTwo$-geometry. 
We will then prove the Orbit Comparison Theorem (Theorem \ref{Main Theorem - Comparison Result}) in Section \ref{Section - Comparison} in a fairly general setup.
Section \ref{Section - EH-Bundle} outlines two different (but equivalent) constructions of a bundle of Eguchi-Hanson spaces.
We will mostly use the second definition (\ref{Eq: Eguchi-Hanson family Version II}) in the remaining text, in particular, we will use it in Section~\ref{Section - Bundle Resolution} to construct bundles of $\GTwo$-manifolds with fibrewise $\GTwo$-structures on them.
As a preparation, we recall in Section~\ref{Section - GeneralisedKummer} Joyce's generalised Kummer construction in a specific setup and prove a divisibility result of their first Pontryagin classes.
Finally, in Section~\ref{Section - Resolution Detection} we prove that our constructed examples give rise to non-trivial in the (homotopy) moduli spaces, which in particular, implies Theorem \ref{Main Theorem - Example} and \ref{Main Theorem - pi_2-HMS}.
Furthermore, in Table \ref{fig - Table Examples}, we list the analogous results for the other examples of Joyce presented \cite{Joyce1996CompactG2II}.

In Appendix \ref{Appendix - Slice Theorem}, we recall properties of the actions of diffeomorphisms on the space of $\GTwo$-structures and outline a proof of Joyce's slice theorem in a formulation that is more suitable for our needs.

\vspace{10pt}
\noindent
\textbf{Acknowledgments:} 
D.~Crowley is partially supported by the Australian Research Council Discovery Project DP DP220102163.
S.~Goette's work was funded by the Simons Foundation, Grant Number 488617.
T.~Hertl's work was funded by the Simons Foundation, Grant Number 488617, and by the Australian Research Council Discovery Project DP DP220102163.
We would like to thank Johannes Nordström and Markus Upmeier for inspiring discussions. Furthermore, T.~Hertl would like to thank MATRIX for their hospitality during the workshop "The Geometry of Moduli Spaces in String Theory" in September 2024, where part of this work was completed.

%% file: FoundationsModuliSpaces.tex
\section{Preliminaries}\label{Section-Foundations}

We will recollect the basic facts about the group $\GTwo$ and manifold with $\GTwo$-holonomy metrics we need.
Nothing presented in this section is original and further details may be found in \cite{Bryant2006Remarks},  \cite{Joyce2000SpecialHolonomy}, \cite{Joyce2007CalibratedGeometry}, and \cite{G2Lectures}, although some of our conventions may differ from these sources.

Write $x^\alpha$ for the standard coordinates of $\R^7$ and abbreviate $\diff x^i\wedge \diff x^j \wedge \dots \wedge \diff x^l$ to $\diff \mathbf{x}^{ij\dots l}$.
Our standard model\footnote{This form corresponds to the one considered in \cite{Bryant1987ExceptionalHolonomy} and \cite{Joyce2000SpecialHolonomy} under the coordinate change $x_1 \mapsto -x_1$, $x_2\mapsto -x_2$.} for a $\GTwo$-structure is the following three form on $\R^7$:
\begin{equation*}
    \GTwostd = \diff \mathbf{x}^{123} - \diff \mathbf{x}^{145} - \diff \mathbf{x}^{167} - \diff \mathbf{x}^{246} + \diff \mathbf{x}^{257} - \diff \mathbf{x}^{347} - \diff \mathbf{x}^{356}. 
\end{equation*}
Under the decomposition $\R^7 = \R^3 \oplus \Quat$ and $\R^7 = \R \oplus \C^3$ this 3-forms splits  into 
\begin{equation}\label{eq - G2FormSplitting}
    \GTwostd = \vol_{\R^3} + \sum_{\alpha = 1}^3 \diff x^\alpha \wedge \omega_\alpha^{\mathrm{std}} = \mathrm{Re}\, \Omega^{\std} + \diff t \wedge \omega^{\std},
\end{equation}
where 
\begin{align}\label{eq - hyper Kaehler forms std}
     \omega_1^{\mathrm{std}} &= \diff y^1 \wedge \diff y^0 + \diff y^3 \wedge \diff y^2, &\omega_2^{\mathrm{std}} = \diff y^2 \wedge \diff y^0 - \diff y^3 \wedge \diff y^1, \\ 
    \omega_3^{\mathrm{std}} &= \diff y^3 \wedge \diff y^0 + \diff y^2 \wedge \diff y^1, \nonumber
\end{align} 
under the identification $y^\alpha = x^{\alpha + 4}$, and
\begin{equation}\label{eq - SU forms std}
    \omega^{\std} = \frac{\iu}{2}\sum_{\alpha=1}^3 \diff z^\alpha \wedge \diff \Bar{z}^\alpha, \qquad \qquad \Omega^{\std} = \diff z^1 \wedge \diff z^2 \wedge \diff z^3
\end{equation}
under the identification $z^\alpha = x^{\alpha} + \iu x^{\alpha + 4}$ and $t = x^4$.
For later, we recall that
\begin{align}
    \ast_{\GTwostd} \GTwostd &= \frac{1}{2} \omega_1^{\mathrm{std}} \wedge \omega_1^{\mathrm{std}} + \omega_1^{\mathrm{std}} \wedge \diff x_2 \wedge \diff x_3 + \omega_2^{\mathrm{std}} \wedge \diff x_3 \wedge \diff x_1 + \omega_3^\mathrm{std} \wedge \diff x_1 \wedge \diff x_2 \\
    &= \diff\mathbf{x}^{4567} - \diff \mathbf{x}^{2367} - \diff \mathbf{x}^{2345} - \diff \mathbf{x}^{1357} + \diff \mathbf{x}^{1346} - \diff \mathbf{x}^{1256} - \diff \mathbf{x}^{1247}. \nonumber
\end{align}

We define the group $\GTwo$ to be the stabiliser of $\GTwostd$ inside $\GL_7(\R)$:
\begin{equation*}
    \GTwo := \left\{ A \in \GL_7(\R) \,:\, A^\ast \GTwostd = \GTwostd \right\}.
\end{equation*}
From the splitting (\ref{eq - G2FormSplitting}), we see that $\Sp(1)$ and $\SU(3)$ are subgroups of $\GTwo$.

The elements of the orbit $\Lambda^{3,+}\R^{7,\vee} = \GL_7(\R)\cdot \GTwostd \cong \GL_7(\R)/\GTwo$ are called \emph{positive 3-forms}.
One can prove by direct computation that 
\begin{equation*}
    6\langle v,w\rangle_{\std} \vol_{\R^7} = (\insertion{v} \GTwostd) \wedge (\insertion{w}\GTwostd) \wedge \GTwostd =: g_{\GTwostd}(v,w)\vol_{\GTwostd}
\end{equation*}
and that this formula is equivariant with respect to the pullback action of $\GL_7(\R)$ on 3-forms and symmetric bilinear forms.
Hence, $\GTwo \leq \SOrth(7)$ and we get an equivariant (non-linear) map
\begin{equation}\label{eq - form to metric}
    \Lambda^{3,+}\R^{7,\vee} \rightarrow S^{2,+}(\R^{7,\vee}), \qquad \vphi \mapsto g_\varphi,
\end{equation} 
which can be identified with the canonical projection $\GL_7(\R)/\GTwo \rightarrow \GL_7(\R)/\SOrth(7)$.

The space of all $\GTwo$-structures on $M$, denoted by
\begin{equation*}
    \GTwoStr(M) = \Gamma(\Lambda^{3,+}T^\vee M) \subseteq \Omega^3(M),
\end{equation*}
is an open subspace of $\Omega^3(M)$ endowed with the smooth Fr\'{e}chet topology.
This space is non-empty if and only if $M$ has a spin structure, that is, if $M$ is orientable and its second Stiefel-Whitney class $w_2(M) \in H^2(M;\Z_2)$ vanishes.
In this case, a choice of $\varphi \in \GTwoStr(M)$ determines a unique spin structure on $M$.
The map (\ref{eq - form to metric}) induces a smooth map $\GTwoStr(M) \rightarrow \Riem(M)$ into the Fr\'{e}chet space of all Riemannian metrics on $M$.
It is surjective, whenever the domain is non-empty.

If $\nabla^{\varphi}$ denotes the Levi-Civita connection of $g_\varphi$, then we call $T(\varphi) := \nabla^\varphi \varphi$ the \emph{torsion tensor} of $\varphi$.
It is proved in \cite[Lemma 11.5]{Salamon1989RedBook} that
\begin{equation*}
    T(\varphi) = 0 \quad \Longleftrightarrow \quad \diff \varphi = 0 \ \text{ and } \  \diff^{\ast_\varphi}\varphi =0, 
\end{equation*}
where $\ast_\varphi$ is the Hodge-$\ast$ operator of the underlying metric $g_\varphi$ and $\diff^{\ast_\varphi}$ is the codifferential with respect to the metric $g_\varphi$, which acts on $k$-forms by $\diff^{\ast_\varphi} = (-1)^k \ast_\varphi^{-1}\diff\,\ast_{\varphi}$.
If $\nabla^\varphi \varphi = 0$, then $\mathrm{Hol}(g_\varphi) \leq \GTwo$, and conversely, by the holonomy principle, if $\mathrm{Hol}(g) \leq \GTwo$, then there is a torsion-free $3$-form $\varphi$ such that $g_\varphi = g$.
In this case, we call $(M,g)$ (often just $M$) a \emph{$\GTwo$-manifold}.
A Riemannian metric with $\Hol(g) \leq \GTwo$ is always Ricci-flat: $\mathrm{Ric}(g)=0$, see \cite[Proposition 11.8]{Salamon1989RedBook}.

For a closed, connected $\GTwo$-manifold, the condition to have full holonomy is a purely topological one, namely $\mathrm{Hol}(g) = \GTwo$ if and only if $\pi_1(M)$ is finite \cite[Proposition 10.2.2]{Joyce2000SpecialHolonomy}.
Thus, in this case, if
\begin{equation*}
    \tfGTwoStr(M) := \{ \varphi \in \GTwoStr(M) \, : \, \diff \varphi = \diff^{\ast_\varphi}\varphi = 0 \} 
\end{equation*}
denotes the space of all torsion-free $\GTwo$-structures and if $\Riem^{\GTwo}(M)$ is the space of Riemannian metric whose holonomy group is $\GTwo$, then we have the following commutative diagram 
\begin{equation}\label{eq - torsionfree G2hol}
    \xymatrix{ \tfGTwoStr(M) \ar@{^{(}->}[r] \ar[d]_{2 : 1}^{\varphi \mapsto g_\varphi} & \GTwoStr(M) \ar[d]^{\varphi \mapsto g_{\varphi}} \\
    \Riem^{\GTwo}(M) \ar[r] & \Riem(M), }
\end{equation}
in which the elements of the fibre of the left vertical map are distinguished by the orientation they induce on $M$.

We close this section with a lemma that is folklore, but we failed to find a proof in the literature.

\begin{lemma}\label{Lemma - finite isometry}
    If $(M,g)$ is a closed Riemannian manifold with $\Hol(g) = \GTwo$, then its isometry group $\Gamma$ is finite. 
\end{lemma}
\begin{proof}
   The isometry group of a closed Riemannian manifold is a compact Lie group. 
   Thus, if $\Gamma$ were not finite, we would find a non-trivial infinitesimal isometry, in other words, a Killing vector field $X \neq 0$.
   The following Bochner formula for vector fields derived in \cite{Yano1952HarmonicandKillingVF},
   \begin{equation*}
       \int_M \mathrm{Ric}(X,X) + \mathrm{tr}\bigl((\nabla_{(\placeholder)}X)^2\bigr) + {\mathrm{tr}(\nabla_{(\placeholder)}X)^2} \diff \vol = 0
   \end{equation*}
   together with $\mathrm{tr}\bigl((\nabla_{(\placeholder)}X)^2\bigr) = - \trace\left( \nabla_{(\placeholder)}X^\ast \nabla_{(\placeholder)}X\right) \leq 0$ and $\mathrm{Ric}(X,X) = 0$ implies that $X$ is a parallel vector field.
   By the holonomy principle, $\Hol_p(g) \leq \mathrm{Stab}_{\GTwo}(X(p)) \cong \mathrm{SU}(3)$ for some $p \in M$ (and hence for each point in the path component of $p$), which is a contradiction. 
\end{proof}

%% file: ComparisonTheorem.tex

\section{Comparing the homotopy moduli space and the moduli space}\label{Section - Comparison}

The aim of this section is to compare the homotopy quotient $\hModuli{M}$ to the actual quotient $\Moduli{M}$ and to give a proof of Orbit Comparison Theorem (Theorem \ref{Main Theorem - Comparison Result}).
It will turn out that this Orbit Comparison Theorem is a consequence of a general principle of continuous action of topological groups on toplogical spaces, so it is worthwhile to establish the machinery in full generality and then apply it to $\GTwo$-geometry. 

Assume that a topological group $G$ acts on a topological space $X$, not necessarily freely.
We can turn this action into a free one without changing the homotopy type of $X$ as follows:
Let $G \rightarrow EG \rightarrow BG$ be the \emph{universal} $G$-principal bundle over the \emph{classifying space} as constructed, for example, by Milnor \cite{Milnor1956UniversalII}, see also \cite[Chapter 14.4]{tomDieck2008algebraic} for a very clear account.
It is a numerable\footnote{A bundle is called \emph{numerable} if the base space has an open covering of trivialising domains and a partition of unity subordinated to that cover.} $G$-principal bundle with contractible total space.
We define the \emph{homotopy quotient} of the action $G \curvearrowright X$ to be the Borel construction
\begin{equation*}
    X/\!\!/G := EG \times_G X = (EG \times X)/G,
\end{equation*}
where $G$ acts diagonally on the product.
The homotopy quotient is again the total space of a numerable fibre bundle $X \rightarrow X/\!\!/G \xrightarrow{\proj_1}BG$, so it induces a long exact sequence of homotopy groups.
By \cite[Proposition 7.9]{may1975classifying} this sequence of maps can be extended by the orbit map to a sequence of quasi-fibrations.
This observation immediately implies the next result.

\begin{lemma}\label{Lemma - Boundary Equals Orbit}
   Under the identification $\pi_n(BG) \cong \pi_{n-1}(\Omega B G) \cong \pi_{n-1}(G)$ via fibre transportation, the boundary operators in the induced long exact sequence of homotopy groups agree with the homomorphism induced by the orbit map. 
   In other words, we have the following commutative diagram
   \begin{equation*}
       \xymatrix{\pi_{n-1}(G,e) \ar@{=}[r] \ar@/_1pc/[rr]_-{[g] \mapsto [g_\bullet x_0]} & \pi_n(BG,\ast) \ar[r]^-{\partial_n}  & \pi_{n-1}(X,x_0). }
   \end{equation*}
\end{lemma}

Recall that $BG$ classifies the homotopy-invariant contravariant functor $\mathsf{Top} \rightarrow \mathsf{Set}$ that assigns to a topological space the set of isomorphism classes of numerable $G$-principal bundles over that space.

In the same spirit, $X/\!\!/G$ also classifies a homotopy invariant functor: If $f \colon B \rightarrow X/\!\!/G$ is a continuous map, we denote the composition of $f$ with $\proj_1\colon X/\!\!/G \rightarrow BG$ by $\mathrm{bs}(f)$.
Essentially by definition, $f$ induces a section into the pullback bundle $\mathrm{bs}(f)^\ast X/\!\!/G \rightarrow B$, which we denote with $\sigma_f$. 
\begin{prop}\label{Prop - Classifying Space}
    Let $B$ be a topological space. 
    The map that assigns to a homotopy class $[f] \in [B,X/\!\!/G]$ the homotopy class $[\sigma_f]$ yields a bijection
    \begin{equation*}
        [B,X/\!\!/G] \rightarrow \bigsqcup_{[P]} \pi_0\left( \Gamma(P \times_G X) \right)\bigr/\sim_{\mathrm{Iso}},
    \end{equation*}
    where $[P]$ runs over set of all isomorphism classes of numerable $G$-principal bundles over $B$ and two homotopy classes of sections satisfy $[\sigma_0] \sim_{\mathrm{Iso}} [\sigma_1]$ if and only if there is a bundle automorphism of $P$ that maps a representative of $[\sigma_0]$ to a representative of $[\sigma_1]$.
\end{prop}
\begin{proof}
    It is fairly easy to see, cf. \cite[{Section 5.6}]{tomDieck2008algebraic}, that two maps $f_1,f_2 \colon B \rightarrow E$ into the total space of a fibration $p \colon E \rightarrow Y$ are homotopic if and only if there is a homotopy $h \colon B \times [0,1] \rightarrow Y$ relating $pf_1$ and $pf_2$ such that the fibre transport $h^\sharp(f_1)$ is fibre-homotopic to $f_2$.

    Continuous maps $f \colon B \rightarrow E$ are in one to one correspondence with sections into the pullback fibration $(pf)^\ast E \rightarrow X$, so $f_1,f_2$ are homotopic if and only if $h^\sharp(f_1)$ and $f_2$ are homotopic sections in $(pf_2)^\ast E$.

     We now deduce the stated claim by applying the previous observations to the fibre bundle $X/\!\!/G \rightarrow BG$ and using the fact that numerable fibre bundles are fibrations \cite[Theorem 13.4.1]{tomDieck2008algebraic} and that the pullback of the universal $G$-principal bundle with a homotopy class induces a bijection between $[B,BG]$ and the set of isomorphism classes of numerable $G$-principal bundles over $B$.
\end{proof}

Finally, the projection to the second component induces a comparison map $X/\!\!/G \rightarrow X/G$.
In general, it is hard to determine the connectivity of this comparison map. 
However, in two special cases we can say something about this map.
The first case is classical: If the action $G \curvearrowright X$ we started with is already free, then the comparison map is a weak homotopy equivalence.

The second case, in which we are more interested in, requires the notion of slices.

\begin{definition}\label{Definition - Slice}
    Let $G \curvearrowright X$ be a continuous group action.
    A pair $(S,\Gamma)$ consisting of a subspace $S \subseteq X$ and a subgroup $\Gamma \leq G$ is called a \emph{slice} if the following conditions are satisfied:
    \begin{itemize}
        \item[(SL1)] $S$ is invariant under the restricted action to $\Gamma$, that is $g \cdot S \subseteq S$ for all $g \in \Gamma$.
        \item[(SL2)] If $g \cdot S \cap S \neq \emptyset$, then $g \in \Gamma$.
        \item[(SL3)] $G \rightarrow G/\Gamma$ is a principal $\Gamma$-bundle and there exists a local section $\chi \colon U \rightarrow G$ around the identity coset such that the map
        \begin{equation*}
            \mu \colon U \times S \rightarrow X, \qquad ([u],s) \mapsto \chi([u])\cdot s
        \end{equation*}
        is an open embedding.
    \end{itemize}
    We equivalently write that $S$ is a $\Gamma$-slice if we wish to emphasise the group $\Gamma$.
    A slice $(S,\Gamma)$ is called \emph{constant} if $\Gamma$ acts on $S$ trivially.
\end{definition}

\begin{rem}\label{Remark - slice consequence}
Note that the three axioms imply that the map $S/\Gamma \rightarrow X/G$ induced by the inclusion $S \hookrightarrow X$ is an open embedding.
Note further that if $(S,\Gamma)$ is a constant slice, then $\Gamma$ is the stabiliser of each $s \in S$ by (SL2), hence $S \rightarrow X/G$ is an open embedding.
\end{rem}

The general version of Theorem \ref{Main Theorem - Comparison Result} is the following result that compares the homotopy quotient to the quotient of a group action if each point has a constant slice.

\begin{theorem}\label{Theorem - HQuoVsQuo}
    Let $G \curvearrowright X$ be a continuous group action such that $X/G$ is connected.
    If every $x \in X$ is contained in a constant slice $(S,\mathrm{Stab}_x(G))$, then the conjugacy class of $\Gamma := \mathrm{Stab}_x(G)$ is independent of $x$ and the comparison map $X/\!\!/G \rightarrow X/G$ is a fibre bundle with fibre $B\Gamma$.
\end{theorem}
\begin{proof}
    Since $\mathrm{Stab}_{gx}(G) = g\mathrm{Stab}_{x}(G)g^{-1}$, the map that assigns to a point $x \in X$ the conjugacy class of its stabiliser is a $G$-invariant map from $X$ into the set of all conjugacy classes of subgroups of $G$.
    Thus, this map will factor through the quotient $X/G$.

    By Remark \ref{Remark - slice consequence} the map $S \rightarrow X/G$ is an open embedding.
    Since the stabiliser is constant on $S$, it follows that conjugacy class of the stabiliser is a locally constant map, hence constant as $X/G$ is connected.    
    
    Recall that $X/\!\!/G = (EG \times X)/G$.
    Fix a model for $EG$. 
    Since $\Gamma$ acts freely on $EG$, we can use $EG/\Gamma$ as a model for $B\Gamma$.
    We regard the $G$-invariant open subset 
    \begin{equation*}
        V=\bigl\{\,gs\bigm|s\in S,g\in G\,\bigr\}\subset X
    \end{equation*}
    as the preimage of~$S\subset X/G$ under the quotient map.

    The map~$\Phi\colon B\Gamma \times S \to EG \times_G V\subset X/\!\!/G$ that sends~$([e],s)$ for~$s\in S$, $e\in EG$ to~$[(e,s)]$ is well-defined, continuous\footnote{Since $\Gamma$ acts trivially on $S$, the obvious map $(EG \times S)/\Gamma \rightarrow (EG/\Gamma) \times S = B\Gamma \times S$ is a homeomorphism.}, 
    bijective, and compatible with the obvious maps to~$S\subset X/G$.

    By~(SL3), the set~$V$ is covered by open subsets of~$X$ of the form~$g\,\mu(U\times S)$ for~$g\in G$. For each~$g$, the inverse of~$\Phi$ pulls back to~$EG \times g\,\mu(U\times S)$ to a continuous map that sends~$\bigl(e, g\,\chi([u])\cdot s\bigr)$ to~$\bigl(s,[e\cdot g \, \chi([u])]\bigr)$.
    Hence~$\Phi$ is a homeomorphism, hence, a local trivialisation of~$X/\!\!/G\to X/G$.
\end{proof}

Returning to $\GTwo$-geometry: We apply the just established theory to $\GTwo$-manifolds by setting $X = \tfGTwoStr(M)$ or $X = \tfGTwoStr(M)_{\varphi_0}$, where the latter denotes the path component of $\varphi_0$ inside $\tfGTwoStr(M)$, and $G = \Diff(M)_0$.
We call $\hModuli{M}$ the \emph{homotopy moduli space}.

\begin{example}
    To get a better feeling for the homotopy moduli spaces, let us consider the example $B= \{\mathrm{pt}\}$ in Proposition \ref{Prop - Classifying Space}.
    In this special case, we have
    \begin{equation*}
        \pi_0\left(\mathcal{G}_2^{\mathrm{tf}}(M)/\!\!/\Diff(M)_0\right) = \pi_0\left(\mathcal{G}_2^{\mathrm{tf}}(M)\right)/\pi_0(\Diff(M)_0) = \pi_0\left(\mathcal{G}_2^{\mathrm{tf}}(M)\right),
    \end{equation*}
    which of course agrees with $\pi_0\left(\mathcal{G}_2^{\mathrm{tf}}(M)/\Diff(M)_0\right)$ as the pullback action of $\Diff(M)_0$ preserves path components.

    In case of the full homotopy quotient $\tfGTwoStr(M)/\!\!/\Diff(M)$ we get
    \begin{equation*}
        \pi_0\left(\mathcal{G}_2^{\mathrm{tf}}(M)/\!\!/\Diff(M)\right) = \pi_0\left(\mathcal{G}_2^{\mathrm{tf}}(M)\right)/\pi_0(\Diff(M)) = \pi_0\left(\mathcal{G}_2^{\mathrm{tf}}(M)/\Diff(M)\right),
    \end{equation*}
    where the last identity is follows from the path-lifting property of $\tfGTwoStr(M) \rightarrow \tfGTwoStr(M)/\Diff(M)$: 
    From Joyce's slice theorem for the full diffeomorphism group, see the footnote above Theorem \ref{Theorem - Slice Theorem}, we get that every element $[\varphi] \in \tfGTwoStr(M)/\Diff(M)$ can be covered by open sets of the form $S'/\mathrm{Stab}(\varphi)$. 
    For every path $\gamma \colon [0,1] \rightarrow \tfGTwoStr(M)/\Diff(M)$, we find a finite partition $0 = a_0 < \dots < a_n = 1$ and open set $S_i'/\mathrm{Stab}(\varphi)$ such that $\gamma([a_i,a_{i+1}]) \subseteq S_i'/\mathrm{Stab}(\varphi_i)$.
    By \cite[Theorem 6.2]{Bredon1972IntroCptTrafoGrps}, all these restrictions $\gamma|_{[a_i,a_{i+1}]}$ can be lifted to $S_i$, therefore $\gamma$ has a lift to $\tfGTwoStr(M)$.
\end{example}

With the established machinery at hand, the proof of Theorem \ref{Main Theorem - Comparison Result}, the Orbit Comparison Theorem, is now rather straightforward.

\begin{proof}[Proof of Theorem \ref{Main Theorem - Comparison Result}]
   By Corollary \ref{Corollary - Slice Torsionfree}, the pullback action of $\Diff(M)_0$ on $\tfGTwoStr(M)$ has a $\mathrm{Stab}_\varphi(\Diff(M)_0)$-slice $S_\varphi$ for each $\varphi$.
   As explained in the proof of \cite[Theorem 10.4.4]{Joyce2000SpecialHolonomy}, $\mathrm{Stab}_\varphi(\Diff(M)_0)$ acts trivially on $S_\varphi$, so the slices are constant.
   Theorem \ref{Theorem - HQuoVsQuo} now implies the result because fibre bundles over paracompact Hausdorff spaces are fibrations.
\end{proof}

\begin{rem}\label{Remark - comparison failure fix points}
    Our proof of Theorem \ref{Main Theorem - Comparison Result} fails for the analogous comparison map for the full quotients $\tfGTwoStr(M)/\!\!/\Diff(M) \rightarrow \tfGTwoStr(M)/\Diff(M)$, which can already be observed for Joyce's first example $M_{\mathrm{Joyce}}$.
    Essentially, this manifold is constructed from a flat orbifold $T^7/\Gamma$, where $\Gamma \cong \Z_2^3$, by desingularising $12$ connected components of the form $T^3$, by replacing the local neighbourhood of the form $T^3 \times \R^4/\Z_2$ by $T^3 \times T^\vee \CP^1$ with a special metric on it, see Section \ref{Section - GeneralisedKummer} below, \cite{Joyce1996CompactG2I} or \cite[§12.2]{Joyce2000SpecialHolonomy} for details.
    It is not hard to write down elements $A \in \GL_7(\R)$ that descend to $\varphi^{\mathrm{std}}$-preserving diffeomorphisms on $T^7/\Gamma$, which permute the components of the singularity set in a non-trivial manner.

    In the case that all "core spheres" $\CP^1 \subseteq T^\vee \CP^1$ have the same radius, then the element $A$ extends to a diffeomorphism on $M$ that still preserves the torsion-free $\GTwo$-form $\tilde{\varphi}$, but acts non-trivial $H_2(M;\R)$ because the "core spheres" are getting permuted in a non-trivial fashion.
    However, if we slightly perturb the radii such that they are all different, then the (new) resulting torsion-free $\tilde{\varphi}_{\mathrm{new}}$ is not preserved by $A$.
    Hence, $\tilde{\varphi}$ is not contained in a constant slice and the assumptions of Theorem \ref{Theorem - HQuoVsQuo} are not satisfied.
\end{rem}

If the closed $\GTwo$-manifold $M$ has finite fundamental group, then the holonomy group of underlying metric $g_\varphi$ is the entire group $\GTwo$, so by Lemma \ref{Lemma - finite isometry} $\mathrm{Stab}_\varphi(\Diff(M)_0)$ is finite. 
Theorem \ref{Main Theorem - Comparison Result} applied to these manifolds gives the following slight strengthening of Corollary \ref{Main Corollary - Coconnectivity}.

\begin{cor}\label{Theorem - Comparison Theorem Detailed}
    If $M$ is a closed $\GTwo$-manifold with finite fundamental group, then
    the comparison map $p \colon \hModuli{M} \rightarrow \Moduli{M}$ induces on homotopy groups based at $[\varphi]$ the following homomorphisms:
    \begin{itemize}
        \item $\pi_k(p)$ is an isomorphism for all $k \geq 3$,
        \item $\pi_2(p)$ is injective and every element in the cokernel has order at most $|\mathrm{Stab}_{\Diff(M)_0}(\varphi)|$,
        \item $\pi_1(p)$ is surjective and every element in the kernel has order at most $|\mathrm{Stab}_{\Diff(M)_0}(\varphi)|$.
    \end{itemize}
\end{cor}

  Since the boundary operator in the long exact sequence of homotopy groups associated to the fibration $\tfGTwoStr(M)_\varphi \rightarrow \tfGTwoStr(M)_\varphi/\!\!/\Diff(M)_0 \rightarrow B\Diff(M)_0$ agrees with the orbit map under the canonical identification, see Lemma \ref{Lemma - Boundary Equals Orbit}, the previous corollary implies the following comparison result.
\begin{cor}
    If $M$ is a closed $\GTwo$-manifold with finite fundamental group, then the canonical homomorphism 
    \begin{equation*}
       \pi_k\left(\tfGTwoStr(M),\varphi\right)/\pi_k(\Diff(M)_0,\id) \rightarrow \pi_k\left(\Moduli{M}, [\varphi]\right)    
    \end{equation*}
    is injective if $k\geq2$.
\end{cor}

%% file: EguchiHansonBundle.tex
\section{The Eguchi-Hanson Bundle}\label{Section - EH-Bundle}

We begin with a short reminder of the Eguchi-Hanson metric.
On $\Quat$, we use the $\Quat$-valued inner product $\langle h_1,h_2\rangle_{\Quat} := h_1\Bar{h}_2$, which is left-linear in the first component and right anti-linear in the second component.
In this way, the inner product is invariant under right multiplication and realises the diffeomorphism $S^3 \cong \Sp(1)$ via $\bigl(q \mapsto (h\mapsto h\bar{q})\bigr)$.
If we decompose this inner product into real and imaginary part, we get
\begin{equation*}
    \langle \placeholder, \placeholder \rangle_\Quat =  \langle \placeholder, \placeholder \rangle_{\R^4}1 +  \iu \omega_{\iu} + \ju  \omega_{\ju} +  \ku \omega_{\ku},
\end{equation*}
which are expressed as (\ref{eq - hyper Kaehler forms std}) in standard coordinates.

To express these forms in polar coordinates, we define, for each $q \in S^3 \cap \mathrm{Im}\, \Quat$, the right-invariant $1$-forms $\sigma_q$ on $\Quat \setminus \{0\}$ via $\sigma_{q,x}(\placeholder) = \mathrm{Re}( qx \Bar{\placeholder} ) / ||x||^2$.
If we set $r = ||x||$, then $(\diff r)_{x} = \mathrm{Re}(x\Bar{\placeholder})/||x||$.
An easy calculation now shows
\begin{align}\label{eq - eucl hyper polar}
    \langle \placeholder,\placeholder \rangle_{\R^4} &= \diff r^2 + r^2\left( \sigma_{\iu}^2 + \sigma_{\ju}^2 + \sigma_{\ku}^2 \right), & \omega_1^{\mathrm{std}} := \omega_{\iu}^{\mathrm{std}} = r \sigma_{\iu} \wedge \diff r + r^2 \sigma_{\ku} \wedge \sigma_{\ju}, \\
    \omega_2^{\mathrm{std}} := \omega_{\ju}^{\mathrm{std}} &= r \sigma_{\ju} \wedge \diff r - r^2 \sigma_{\ku} \wedge \sigma_{\iu}, & \omega_3^{\mathrm{std}} := \omega_{\ku}^{\mathrm{std}} = r\sigma_{\ku} \wedge \diff r + r^2 \sigma_{\ju} \wedge \sigma_{\iu}. \nonumber
\end{align}

Eguchi and Hanson in \cite[p.~91ff]{Eguchi1979SelfDual} showed that the Riemannian metrics on $\Quat/\Z_2 \setminus \{0\}$ given by
\begin{equation}\label{eq - EguchiHanson metric}
    \gEHpar{a} = \left( 1 + (a/r)^4 \right)^{-1/2} \left( \diff r^2 + r^2 \sigma_{\iu}^2\right) + \left( 1 + (a/r)^4 \right)^{1/2} \left( r^2\sigma_{\ju}^2 + r^2\sigma_{\ku}^2 \right)
\end{equation}
parametrised by $a > 0$, extends, under the identification $T^\vee \CP^1 \setminus \CP^1 \cong \Quat/\Z_2 \setminus \{0\}$, to a Riemannian metric on 
\begin{equation*}
    T^\vee \CP^1 = S^3 \times_{S^1,(\placeholder)^2}\C = (S^3 \times \C)/\sim,
\end{equation*}
where $(p,\lambda) \sim (\eu^{\iu\theta}p,\eu^{-2\iu\theta}\lambda)$.
More precisely, the \emph{desingularisation map} $\sigma \colon T^\vee \CP^1 \rightarrow \Quat/\Z_2$ given by $[p,\lambda] \mapsto [\sqrt{\lambda|\lambda|\,}p]$ restricts on the complement of $\CP^1 = \sigma^{-1}(0)$ to a diffemorphism, and $\sigma^\ast g_{EH,a}$ extends to a hyperkähler metric that we also denote with $g_{EH,a}$.
We wish to emphasise here that we used the maximal torus $\{\mathrm{e}^{\iu \theta} \, : \, \theta \in \R\}$ whose infinitesimal generator corresponds to the distinguished direction $\sigma_{\iu}$.
Since $f_a(r) := (1+(a/r)^4)^{1/4} = 1 + \bigO(r^{-4})$ for $r \to \infty$, this metric is asymptotically, locally Euclidean, in particular, complete.

Eguchi and Hanson further showed that the Levi-Civita connection $1$-form expressed in the orthonormal coframe $(e^\alpha)_{0 \leq \alpha \leq 3} = (f_a^{-1} \diff r, f_a^{-1} r\sigma_{\iu}, rf_a \sigma_{\ju}, rf_a\sigma_{\ku})$ is self-dual\footnote{In \cite{Eguchi1979SelfDual}, the authors require the forms to be anti-self dual due to their opposite orientation convention.} turning $T^\vee \CP^1$ into a complete hyperkähler manifold. 

The forms defined on $\Quat/\Z_2 \setminus \{0\}$ via
\begin{align}\label{eq - EguchiHanson forms}
    \omega^{\EH,a}_{1}:= \omega^{\EH,a}_{\iu} &= e^{1} \wedge e^0 + e^3 \wedge e^2, & \omega^{\EH}_{2}:= \omega^{\EH}_{\ju} = e^{2} \wedge e^0 - e^3 \wedge e^1 = \omega_{\ju}^{\std}, \\
    \omega^{\EH,a}_{3}:= \omega^{\EH,a}_{\ku} &= e^3 \wedge e^0 + e^2 \wedge e^1 = \omega_{\ku}^{\std}, \nonumber
\end{align}
are parallel with respect to the Levi-Civita connection of the Eguchi-Hanson metric and satisfy $||\omega_q^{\EH,a} - \omega^{\std}_q||_{\mathrm{eucl}} =\bigO(r^{-4})$.
Furthermore, $\omega_1^{\EH} - \omega_1^{\std} = \diff \tau_1^a$ with $\tau_1^a = 1/2 \cdot r^2f_a^2(r) \sigma_{\iu}$, which obviously satisfies $||\nabla^k \tau_1^a|| = O(r^{-(3+k)})$ for $r \to \infty$ because $||\sigma_i||_{g_{EH,a}} = O(r^{-1})$ for $r \to \infty$.
It follows by uniqueness that these forms also extend to $T^\vee \CP^1$ and form the hyperkähler structure there.

Fibre-multiplication with $a$ induces an isometry 
$(T^\vee \CP^1,a^2g_{EH,a^{-1}}) \rightarrow (T^\vee \CP^1,g_{EH,1})$, so we will drop the index $a$ from the notation, unless we wish to particularly emphasise it, and we abbreviate the \emph{Eguchi-Hanson space} $\bigl( T^\vee \CP^1,\gEH, \omega^{\EH}_{1},$ $\omega^{\EH}_{2},$$\omega^{\EH}_{3}  \bigr)$ simply to $EH$.

In complete analogy to (\ref{eq - hyper Kaehler forms std}), the hyperkähler-structure on $EH$ induces a $\GTwo$-structure on $T^3 \times EH$ as follows:
We use the interpretation $T^3 = \{ \mathrm{diag}(\eu^{\iu\theta_1},\eu^{\ju\theta_2}, \eu^{\ku\theta_3}) \, : \, \theta_\alpha \in \R \}$, so that $T_{\id}T^3 = \mathrm{Im} \Quat$. 
The corresponding right-invariant $1$-forms are denotes by $\delta^q = \langle q, \placeholder\rangle_{\R^4} = \mathrm{Re}(q \cdot \bar{\placeholder})$.
The $\GTwo$-structure is now given by
\begin{equation}\label{eq - GTwo form EH}
    \varphi^{\EH,a} = \delta^{\iu} \wedge \delta^{\ju} \wedge \delta^{\ku} + \sum_{\alpha \in \{ \iu,\ju,\ku\}} \delta^{\alpha} \wedge \omega_\alpha^{\EH,a}.
\end{equation}

\subsection{Variation of the Complex Structure}\label{Subsection - Variation of the complex structure}

Equation (\ref{eq - EguchiHanson metric}) shows that the Eguchi-Hanson metric is less symmetric than the euclidean metric. 
Indeed, the group of all orientation preserving isometries of $g_{EH}$ is given by $\Orth(2)\cdot\Sp(1)/\Z_2 \cong \Orth(2) \times \SOrth(3)$ where $\Orth(2)$ is realised by quaternionic left multiplication with elements of the form $\mathrm{e}^{\iu \theta}$ or $\mathrm{e}^{\iu \theta}\ju$ and $\Sp(1) = \{ (\placeholder)\bar{q} \, : \, q \in S^3 \}$ as explained above, while the isometry group of $\Quat/\Z_2$ is $\SOrth(4)/\Z_2 =  S^3\cdot \Sp(1)/\Z_2 \cong \SOrth(3) \times \SOrth(3)$. 
We write an element of $\SOrth(4) = S^3 \cdot \Sp(1)$ as $q_1 \cdot q_2$.
The group holomorphic isometries of $EH$, i.e. $\mathrm{Stab}(g^{EH},\omega^{EH}_\iu)$ is the subgroup $S^1 \cdot \Sp(1)/\Z_2 \cong S^1 \times \SOrth(3)$, see, for example, \cite{platt2022g2instantons} for the two statements.

The choice of $\iu$ as the distinguished direction is completely arbitrary. 
In fact, Eguchi and Hanson in \cite{Eguchi1979SelfDual} chose $\ku$ to be the distinguished one.
To eliminate these choices, we will consider all possible choices at once. 
This leads to the following bundle:
\begin{equation}\label{eq - Eguchi Hanson Bundle Symmetry Defect Def}
    \mathfrak{EH} := \SOrth(4) \times_{S^1\cdot \Sp(1)} EH \,  \xrightarrow{\mathrm{pr}_1} \, \SOrth(4)/S^1  \cdot \Sp(1) = \CP^1, 
\end{equation}
where $S^1 \cdot \Sp(1)$ acts diagonally from the left on the two factors. 
The desingularisation map $\sigma \colon EH \rightarrow \Quat/\Z_2$ given by $[p,\lambda] \mapsto [\sqrt{\lambda|\lambda|\,}p]$ extends to a map of fibre bundles
\begin{equation}\label{eq - desingularisation map symmetry defect version}
    \boldsymbol{\sigma} \colon \mathfrak{EH} \rightarrow \Quat/\Z_2 \times \CP^1, \qquad \qquad [q_1 \cdot q_2,[p,\lambda]] \mapsto ([\bar{q_1}\sqrt{\lambda|\lambda|\,}p q_2], [q_1]) 
\end{equation}
and is a diffeomorphism on the complement of the fibrewise zero section $\iota \colon \CP^1 \times \CP^1 \rightarrow \mathfrak{EH}$ that is given by $\iota([p],[q]) = [q \cdot 1,[p,0]]$.
We wish to emphasise for later purpose that the target of $\bm{\sigma}$ is a trivial bundle while the domain is not.

The Eguchi-Hanson metric easily extends to a fibre metric $\boldsymbol{g}^{\mathfrak{EH}} = \{g^{\mathfrak{EH}}_{[q]}\}_{[q] \in \CP^1}$ on $\mathfrak{EH}$, i.e. to a section of $\mathrm{S}^2(T^{\vertical,\vee} \mathfrak{EH})$. 
Indeed, under the canonical identification $\mathrm{S}^2(T^{\vertical,\vee} \mathfrak{EH})$ with $\SOrth(4) \times_{S^1 \cdot \Sp(1)} \mathrm{S}^2 T^\vee EH$, the fibre metric can be written as $g^{\mathfrak{EH}}_{[q]} = (q \cdot q_2, g^{EH})$.

This expression is well defined because two different elements of $[q] \in \CP^1$ in $\SOrth(4)$ differ by an isometry of $g^{EH}$.
In fact, $g^{\mathfrak{EH}}$ is even a fibrewise Kähler metric with fibrewise Kähler form $\omega^{\mathfrak{EH}}_{1,[q]} = (q \cdot q_2, \omega_1^{EH})$.
Away from the fibrewise zero section, the metric takes the form 
\begin{equation*}
    \boldsymbol{\sigma}_{[q]\, \ast}(g^{\mathfrak{EH}_{[q]}}) = \left( 1 + (a/r)^4 \right)^{-1/2} \left( \diff r^2 + r^2 \sigma_{q^{-1}\iu q}^2\right) + \left( 1 + (a/r)^4 \right)^{1/2} \left( r^2\sigma_{q^{-1}\ju q}^2 + r^2\sigma_{q^{-1} \ku q}^2 \right)
\end{equation*}
and the Kähler form takes the form
\begin{equation*}
    \boldsymbol{\sigma}_{[q]\, \ast}(\omega^{\mathfrak{EH}}_{1,[q]}) = f_a^{-2}r \left(\sigma_{q^{-1}\iu q} \wedge \diff r\right) + f_a^2 r^2 \left(\sigma_{q^{-1}\ku q} \wedge \sigma_{q^{-1}\ju q}\right)
\end{equation*}
from which we see that $\omega^{\mathfrak{EH}}_{1,[q]} - \ell_q^\ast\omega_1^{\mathrm{std}} = 1/2 r^2 (1+(a/r)^{4})^{1/2}\sigma_{q^{-1}\iu q} = \diff \, \ell_q^\ast \tau_1$.

On the other hand, the hyperkähler structures on each fibre do not assemble to global object although $g^{\mathfrak{EH}}_{[q]}$ is a hyperkähler metric on each fibre.
Indeed, we will see in Lemma \ref{Lemma - Char Clases Eguchi Hanson family} below that $c_1(T^\vertical \mathfrak{EH}) \neq 0$. 

Nevertheless, we still can define on $T^3 \times \mathfrak{EH}$ a fibrewise torsion-free $\GTwo$-structure $\mathbf{\varphi}^{\mathfrak{EH}} = \{\varphi_{[q]}^{\mathfrak{EH}}\}_{[q] \in \CP^1}$ generalising construction (\ref{eq - GTwo form EH}) by
\begin{equation}
    \varphi^{\mathfrak{EH}}_{[q]} = [ q \cdot q_2, \delta^{\iu} \wedge \delta^{\ju} \wedge \delta^{\ku} + \sum_{\alpha \in \{\iu, \ju, \ku\}}  \delta^{q^{-1}\alpha q} \wedge \omega_\alpha^{EH}],
\end{equation}
where we used again the identification $\Lambda^3 T^{\vertical,\vee} \mathfrak{EH} = \SOrth(4) \times_{S^1\cdot \Sp(1)} \Lambda^3 T^{\vertical,\vee}EH$.

To see that this definition is well defined, recall that $\omega^{EH}_{\ju} = \omega^{\mathrm{std}}_{\ju}$ and $\omega^{EH}_{\ku} = \omega^{\mathrm{std}}_{\ku}$, which implies
\begin{align*}
    &\quad \ ( \mathrm{e}^{\iu \theta} \cdot q_2)^\ast \bigl( \delta^{\iu} \wedge \delta^{\ju} \wedge \delta^{\ku} + \sum_{\alpha \in \{\iu, \ju, \ku\}}  \delta^{\mathrm{e}^{-\iu \theta} \alpha \mathrm{e}^{\iu \theta}} \wedge \omega_\alpha^{EH} \bigr) \\
    &=\delta^{\iu} \wedge \delta^{\ju} \wedge \delta^{\ku} + \sum_{\alpha \in \{\iu, \ju, \ku\}}  \delta^{\mathrm{e}^{-\iu \theta} \alpha \mathrm{e}^{\iu \theta}} \wedge \ell^\ast_{\mathrm{e}^{\iu \theta}} r^\ast_{q_2}\omega_\alpha^{EH} \\
    &= \delta^{\iu} \wedge \delta^{\ju} \wedge \delta^{\ku} + \delta^{\mathrm{e}^{-\iu \theta} \iu \mathrm{e}^{\iu \theta}} \wedge \omega_{\iu}^{\mathrm{std}} + \sum_{\alpha \in \{\ju, \ku\}}  \delta^{\mathrm{e}^{-\iu \theta} \alpha \mathrm{e}^{\iu \theta}} \wedge \ell^\ast_{\mathrm{e}^{\iu \theta}} \omega_\alpha^{\mathrm{std}}\\
    &= \delta^{\iu} \wedge \delta^{\ju} \wedge \delta^{\ku} + \sum_{\alpha \in \{\iu, \ju, \ku\}}  \delta^{ \alpha } \wedge \omega_\alpha^{EH}.
\end{align*}

In complete similarity to the fibre metric and the fibre Kähler form, the fibrewise $\GTwo$-form away from the fibrewise zero section takes the form
\begin{equation*}
    \bm{\sigma}_{[q] \, \ast}(\varphi^{\mathfrak{EH}}_{[q]}) = \delta^{\iu} \wedge \delta^{\ju} \wedge \delta^{\ku} + \sum_{\alpha \in \{\iu, \ju, \ku\}}  \delta^{q^{-1}\alpha q} \wedge \ell_q^\ast \omega_\alpha^{EH},
\end{equation*}
and since the standard $\GTwo$-structure is $\SOrth(4)$-invariant we deduce that 
\begin{align*}
    \left|\left| \bm{\sigma}_{[q]\, \ast} \bm{\varphi}^{\mathfrak{EH}}_{[q]} - \varphi^{\mathrm{std}} \right|\right|_{\mathrm{eucl}} = \left|\left| \sum_{\alpha \in \{\iu, \ju,\ku\}} \delta^\alpha \wedge (\omega_\alpha^{EH} - \omega_\alpha^{\mathrm{std}}) \right|\right|_{\mathrm{eucl}} = \left|\left| \delta^{\iu} \wedge \diff \tau_1 \right|\right| \leq C/r^4,
\end{align*}
away from a compact neighbourhood of the fibrewise zero section.

\subsection{Topological Properties}\label{subsection - Topological Properties of the Eguchi-Hanson family}

To read off topological properties of the Eguchi-Hanson family $\mathfrak{EH}$, in particular the characteristic classes of its vertical tangent bundle, another description of $\mathfrak{EH}$ turns out to be more convenient.
To this end, we define 
\begin{equation}\label{Eq: Eguchi-Hanson family Version II}
    \mathcal{EH} = (S^3 \times \CP^1) \times _{S^1,(\placeholder)^2} \C = (S^3 \times \CP^1 \times \C)/\sim,
\end{equation}
where $(p,[q],\lambda) \sim (q^{-1}\eu^{\iu\theta}qp,[q],\eu^{-2\iu\theta}\lambda)$.
The projection given by
\begin{equation*}
    \mathcal{EH} \rightarrow \CP^1 \times \CP^1, \qquad \qquad [p,[q],\lambda] \mapsto ([qp],[q])
\end{equation*}
turns $\mathcal{EH}$ into a $\C$-vector bundle over $\CP^1 \times \CP^1$, while the projection to the second component turns $\mathcal{EH}$ into a fibre bundle over $\CP^1$ with fibre $T^\vee \CP^1$.
This fibre bundle also comes with a family of resolution maps that is given by
\begin{equation*}
    \bm{\sigma}\colon \mathcal{EH} \rightarrow \Quat/\Z_2 \times \CP^1, \qquad \qquad [p,[q],\lambda] \mapsto \bigl([\bar{q}\sqrt{\lambda|\lambda|\,}q p],[q]\bigr)
\end{equation*}
We wish to emphasise for later purpose that target is a trivial bundle while the domain is not.
\begin{lemma}\label{lemma - Variants of Eguchi Hanson family agree}
    The two definitions of the Eguchi-Hanson family agree. 
    More precisely, there is a fibre diffeomorphism that is compatible with the resolution maps.
\end{lemma}
\begin{proof}
    We denote the family of resolution maps of with $\mathfrak{EH}$ with $\bm{\sigma}^{\mathfrak{EH}}$.
    It is readily verified that the map
    \begin{equation*}
        \Phi \colon \mathcal{EH} \rightarrow \mathfrak{EH}, \qquad \qquad [p,[q],\lambda] \mapsto [q \cdot 1, [qp,\lambda]]
    \end{equation*}
    is well-defined, smooth, compatible with the bundle projections, and invertible with smooth inverse
    \begin{equation*}
        \Psi \colon \mathfrak{EH} \rightarrow \mathcal{EH}, \qquad \qquad [q_1 \cdot q_2, [p,\lambda]] \mapsto [\bar{q}_1 p q_2, [q_1],\lambda].
    \end{equation*}
    An explicit calculation now shows that $\bm{\sigma}^{\mathcal{EH}} = \bm{\sigma}^{\mathfrak{EH}} \circ \Phi$ and the claim is proved.
\end{proof}
\begin{cor}
    The bundle $T^3 \times \mathcal{EH} \rightarrow \CP^1$ carries a fibre-wise torsion-free $\GTwo$-structure $\bm{\varphi}^{\mathcal{EH}} = \{\varphi_{[q]}^{\mathcal{EH}}\}_{[q]\in \CP^1}$ such that $||\sigma_{[q]\, \ast}\varphi_{[q]}^{\mathcal{EH}} - \GTwostd||_{\mathrm{eucl}} = \mathcal{O}(r^{-4})$ on the complement of $B_r(0)/\Z_2$.
\end{cor}
\begin{proof}
    Set $\bm{\varphi}^{\mathcal{EH}} = \Phi^\ast \bm{\varphi}^{\mathfrak{EH}}$, then the result follows from Lemma \ref{lemma - Variants of Eguchi Hanson family agree} and the results of the previous subsection.
\end{proof}

Let us now use the just established second description of the Eguchi-Hanson family to study its characteristic classes.
\begin{lemma}\label{Lemma - Char Clases Eguchi Hanson family}
    Under the identification $H^\ast(\CP^1 \times \CP^1) \cong \Z [x,y]/\langle x^2,y^2\rangle$, we have that
    \begin{itemize}
        \item[(i)] the first Chern class of the complex vector bundle $\mathcal{EH} \rightarrow \CP^1 \times \CP^1$ is given by $c_1(\mathcal{EH}) = 2(x-y)$, and that 
        \item[(ii)] the first Pontryagin class of the vertical tangent bundle satisfies $p_1(T^{\vertical}\mathcal{EH}) = -8xy$.  
    \end{itemize}
\end{lemma}
\begin{proof}
    Clearly, under the inclusion $\iota_1\colon \CP^1 = \CP^1 \times \{[1]\} \hookrightarrow \CP^1 \times \CP^1$, the bundle $\mathcal{EH}$ pulls back to $T^\vee \CP^1 = \mathcal{O}(-2)$.
    On the other hand, if $\Delta \colon \CP^1 \rightarrow \CP^1 \times \CP^1$ denotes the diagonal map, then the map
    \begin{equation*}
       \underline{\C} := \CP^1 \times \C \rightarrow \Delta^\ast \mathcal{EH}, \qquad \text{ given by} \qquad ([q],\lambda) \mapsto \bigl( [1, [q],\lambda], ([q],[q]) \bigr)
    \end{equation*}
    is a vector bundle isomorphism.
    
    Under the identification $H^2(\CP^1 \times \CP^1) \cong \mathrm{span}_{\Z}\{x,y\}$ and $H^2(\CP^1) = \mathrm{span}_{\Z}\{c\}$, the induced homomorphism $H^2(\iota_1)$ sends $x$ to $c$ and $y$ to zero, while $H^2(\Delta)$ sends the two generators $x$ and $y$ to $c$.
    Hence, we deduce from naturality of Chern classes and 
    \begin{equation*}
        c_1(\iota_1^\ast \mathcal{EH}) = c_1(T^\vee\CP^1) = 2c \qquad \text{and} \qquad c_1(\Delta^\ast \mathcal{EH}) = c_1(\underline{\C}) = 0 
    \end{equation*}
    that $c_1(\mathcal{EH}) = 2(x-y)$, which proves $(i)$.
    \vspace{4pt}

    Since the zero section 
    \begin{equation*}
      \iota \colon \CP^1 \times \CP^1 \hookrightarrow \mathcal{EH} \qquad \text{given by} \qquad \left([p],[q]\right) \mapsto ([ q^{-1}p, [q], 0])    
    \end{equation*}
    is a homotopy equivalence, the first Pontryagin class of $T^{\vertical}\mathcal{EH}$ is completely determined by the first Pontryagin class of the restricted bundle $\mathcal{EH}|_{\CP^1 \times \CP^1}$.
    The vertical tangent bundle of this restriction can be decomposed into the fibrewise tangent bundle of the zero-section and its normal bundle:
    \begin{equation*}
        T^\vertical \mathcal{EH}|_{\CP^1 \times \CP^1} \cong \mathrm{pr_1}^\ast T\CP^1 \oplus \mathcal{EH}.
    \end{equation*}
    Since $H^4(\CP^1 \times \CP^1) \cong \Z$ is torsion-free, the first Pontryagin class is additive under Whitney sums of vector bundles, so we deduce the claimed formula by
    \begin{equation*}
        p_1(T^\vertical\mathcal{EH}|_{\CP^1\times \CP^1}) = p_1(\mathrm{pr}_1^\ast T^\vee \CP^1 ) + p_1(\mathcal{EH}) = 0 +c_1(\mathcal{EH})^2 = -8xy.\qedhere
    \end{equation*}
\end{proof}

%% file: GeneralisedKummer.tex
\section{Generalised Kummer Construction}\label{Section - GeneralisedKummer}

In this section, we will recall the generalised Kummer construction following Joyce \cite{Joyce1996CompactG2II} and Platt \cite{platt2022improvedEstimates}, which gives us the opportunity to fix some notation and conventions. 
Most of the presented results here are not new, except the result on the divisibility of the Pontryagin class in Theorem \ref{Theorem - DivPontrClass}.

The initial data of a generalised Kummer construction is given by a finite group $\Gamma$ that acts on $(T^7,\GTwostd)$ in $\GTwostd$-preserving manner.
The $\GTwo$-structure then descends to the quotient $T^7/\Gamma$ turning it into a flat orbifold.
The singularity set of this orbifold is generated by the fix points of the action of $\Gamma$.

In the following, we will restrict ourselves to those group actions $\Gamma \curvearrowright T^7$ that additionally satisfy the following conditions.

\begin{Conditions}\label{Conditions - GroupsAction}
$ $
 \begin{enumerate} 
     \item $\pi_1(T^7/\Gamma)$ is finite.
     \item Each component of the singular set of $T^7/\Gamma$ can be resolved by Eguchi-Hanson spaces in the following sense:
     For each component $a$ of the singularity set $S$ in $T^7/\Gamma$, there is an open neighbourhood around it that is isometrically isomorphic to $(T^3 \times B^4_\zeta(0)/\Z_2)/F_a$, where $\zeta > 0$ is a sufficiently small, $F_a \leq \SOrth(4)/\Z_2$ is a finite subgroup of order prime to $3$.
     We require that $F_a$ acts diagonally on the product. 
     On $T^3 = T^3 \times \{[0]\}$, it shall act freely and isometrically, while it shall act on $\Quat/\Z_2$ in the tautological fashion.  
     Furthermore, we require that there is a subgroup $\tilde{F}_a \leq \Orth(2) \times \SOrth(3) = \mathrm{Iso}^+(EH,g_{EH})$ isomorphic to $F_a$ that consists of orientation preserving isometries, and that there exists an equivariant diffeomorphism $\Phi_a \colon EH\setminus \CP^1 \rightarrow \Quat/\Z_2 \setminus \{[0]\}$ that decomposes into $\sigma$ and an orientation preserving isometry $I_a \in \SOrth(4)$: 
     \begin{equation*}
         \xymatrix{\Phi_a \colon EH \ar[r]^{\sigma} & \Quat/\Z_2 \ar[r]^{I_a} & \Quat/\Z_2. } 
     \end{equation*}
     \item For at least one component, we have $F = \{1\}$. 
 \end{enumerate}
\end{Conditions}

We wish to remark that if $F_a \leq \U(2)/\Z_2 = \mathrm{Stab}(EH,g_{EH},\omega_1^{EH})$, then we can choose $\tilde{F}_a = F_a$ and $\Phi_a = \sigma$ because $\sigma$ is $\U(2)/\Z_2$-equivariant. 
However, we are not bound to make this particular choice\footnote{We will always choose $\sigma$ as the resolution map if $F = \{1\}$, though.}, as the next example, which is a regular theme in \cite{Joyce1999ExceptionalHolonomy} and \cite{Joyce1996CompactG2II}, illustrates.

\begin{example}\label{Example - Choices of Desingularisations}
    Assume $F = \Z_2$ acts on $\C^2/\Z_2$ via $[z_1,z_2] \mapsto [z_1,- z_2]$.
    Under the identification $\C^2 \cong \Quat$ via $(z_1,z_2) \mapsto z_1 + z_2\ju$, this involution corresponds to the conjugation with $\iu$.

    There are essentially two ways to regularise this singularity:
    \begin{itemize}
        \item[(i)] Since $F \leq \U(2)/\Z_2$, it extends immediately to $EH$, that is, we can choose $\tilde{F} = F$ and $\Phi = \sigma$. 
        This approach corresponds to what \cite{Joyce1999ExceptionalHolonomy} refers to as using a 'crepant resolution' for desingularisation.
        \item[(ii)] We choose $\tilde{F} = \{\id,\ju\}$, where $\ju$ acts on $EH$ as
        \begin{equation*}
            [p,\lambda] \mapsto [-\ju p\ju,\bar{\lambda}],
        \end{equation*}
        then $\tilde{F}$ is a subgroup of $\mathrm{Iso}^+(EH,g_{EH}) \setminus \mathrm{Stab}(g_{EH},\omega_1)$.
        In this case, the map $\Phi \colon EH \rightarrow \Quat/\Z_2$ given by
        \begin{equation*}
            [p,\lambda] \mapsto  [\mp \ 1/2 \cdot (\iu+\ju) \sqrt{\lambda|\lambda|\,} p  (\iu+\ju) ] 
        \end{equation*}
        satisfies the required equivariance property, and it is a diffeomorphism away from the singularity and its preimage.
        
        Note that $\sigma$ is equivariant with respect to the conjugation with $\ju$, but on $\Quat/\Z_2$ the conjugation with $\ju$ corresponds to the action $(z_1,z_2) \mapsto (\bar{z}_1,\bar{z}_2)$.
        Thus, this approach corresponds to what Joyce in \cite{Joyce1999ExceptionalHolonomy} calls 'smoothing'.
    \end{itemize}
\end{example}

Of course, since the singularity set consists of finitely many components, we may assume that $\zeta < 1$ was chosen uniformly over all components and is small enough such that these open subset are pairwise disjoint.
Condition \ref{Conditions - GroupsAction} yields, for each component of the singular set $S$, an $R$-datum in the sense of Joyce \cite{Joyce2000SpecialHolonomy}, so this orbifold can be resolved by the generalised Kummer construction to a manifold $M$ that carries a parallel $\GTwo$-structure.
Since we will modify this construction in the next section, we are going to recall the necessary parts here.

On the topological side, the generalised Kummer construction produces a manifold $M$ that is obtained by the following pushout
\begin{equation}\label{eq - M as pushout}
    \xymatrix{U  \ar@{=}[d] &&& V \ar@{=}[d] \\
    \bigsqcup_{a \in \pi_0(S)} \bigl( T^3 \times B^4_\zeta(0)/\Z_2 \setminus \{[0]\} \bigr)/F_a \ar@{^{(}->}[d]   \ar[rrr]^-{\id \times \Phi_a^{-1}} &&& \bigsqcup_{a \in \pi_0(S)} \bigl( T^3 \times EH_{< \zeta} \bigr)/\tilde{F}_a  \ar[d]  \\ 
      T^7/\Gamma \setminus S \ar[rrr] &&& M,  }
\end{equation}
were $EH_{< \zeta}$ denotes the preimage of $B_{\zeta}^4(0)$ under $\sigma$.

Next, we are going to recall the construction of the torsion-free $\GTwo$-form.
We restrict our presentation to the construction of the pre-perturbed $\GTwo$-structure for the singularities with $F_a = \{1\}$, so that we have $\Phi_a = \sigma$ by assumption, because we will only resolve these kinds of singularities in a parametrised manner in Section \ref{Section - Bundle Resolution}.
For more details and the general case, we refer to \cite{platt2022improvedEstimates} or the original 
\cite{Joyce1996CompactG2II} and its generalisation in \cite{Joyce2000SpecialHolonomy}.

By the pushout property, there is a unique map $q \colon M \rightarrow T^7/\Gamma$, the \emph{resolving map}, that is the canonical inclusion on $T^7/\Gamma \setminus S$ and the map $\id \times \sigma$ on $S \times EH_{< \zeta}$.
This map is even smooth.
 
The pushout property further implies that there is a unique continuous function $\Check{r} \colon M \rightarrow \R_{\geq 0}$ such that it restricts to $(s,v) \mapsto |\sigma(v)|$ on $S \times EH_{< \zeta}$ and, on $T^7/\Gamma \setminus S$, it restricts to the map
\begin{align*}
    [y] \mapsto \begin{cases}
       |x|, & \text{if } [y]=(\tau,[x]) \in U, \\ 
        \zeta, & \text{if } [y] \notin U.
    \end{cases}
\end{align*}
Fix a smooth, monotonically increasing function $\chi\colon [0,\zeta^2] \rightarrow [0,1]$ such that $\chi(s) = 0$ if $s \leq \zeta^2/4$ and $\chi(s) = 1$ if $s \geq 3\zeta^2/4$.
Using the more convenient index set $\{\iu,\ju,\ku\}$ for $\{1,2,3\}$, we define on $EH_{< \zeta}$ the hyperkähler triple $(\Tilde{\omega}_{\iu}^t, \Tilde{\omega}_{\ju}^t, \Tilde{\omega}_{\ku}^t)$ with 
\begin{equation*}
    \Tilde{\omega}_{\iu}^t := \omega_{\iu}^{EH,t} - \diff(\chi(\Check{r}^2) \tau_{\iu}^t ), \qquad \Tilde{\omega}_{\ju}^t = \omega_2^{EH,t}, \qquad  \Tilde{\omega}_{\ku}^t = \omega_{\ku}^{EH,t}.
\end{equation*}
This hyperkähler triple agrees with the hyperkähler triple of the Eguchi-Hanson space $EH$ in a neighbourhood of the zero section $\CP^1$ and with the Euclidean hyperkähler triple near the boundary $T^3 \times \zeta \RP^3$.
These forms are obviously closed.  

Together with a parallel orthonormal coframe $\delta^{\iu},\delta^{\ju},\delta^{\ku}$ on $T^3$, we define, for each connected component $S_a \cong T^3$ of the singularity set $S$ with $F_a = \{1\}$, the forms on $T^3 \times EH_{< \zeta}$ via
\begin{align}
    \varphi^{t} &:= \delta^{\iu} \wedge \delta^{\ju} \wedge \delta^{\ku} + \sum_{\alpha \in \{\iu,\ju,\ku\}} \delta^{\alpha } \wedge \Tilde{\omega}_\alpha^t = \varphi^{EH,t} - \delta^{\iu } \wedge \diff( \chi(\Check{r}^2) \tau_{\iu }^t)
    \label{eq: preperturbed GTwo form} \\ 
    \begin{split}
        \vartheta^{t} &:= \frac{1}{2}\cdot \Tilde{\omega}_{\iu}^t \wedge \Tilde{\omega}_{\iu}^t + \Tilde{\omega}_{\iu}^t \wedge \delta^{\ju} \wedge \delta^{\ku} + \Tilde{\omega}_{\ju}^t \wedge \delta^{\ku} \wedge \delta^{\iu } + \Tilde{\omega}_{\ku}^t \wedge \delta^{\iu} \wedge \delta^{\ju}
    \end{split} \label{eq: preperturbed first integral}\\
    \begin{split}&\, = \ast_{\varphi^{EH,t}}\varphi^{EH,t} -\omega_{\iu}^{\EH,t} \wedge \diff (\chi(\Check{r}^2)\tau_{\iu}^t) + 1/2 \cdot \diff (\chi(\Check{r}^2)\tau_{\iu}^t) \wedge \diff (\chi(\Check{r}^2)\tau_{\iu}^t)\\
    &\qquad \qquad - \diff (\chi(\Check{r}^2)\tau_{\iu}^t) \wedge \delta^{\ju} \wedge \delta^{\ku}\end{split} \nonumber \\ 
    \psi^{t} &:= \varphi^{t} - \ast_{\varphi^{t}} \vartheta^{t}. \label{eq: preperturbed error}
\end{align} 
 
The form $\varphi^{t}$ agrees with the standard $\GTwo$-structure $\GTwostd$ near the boundary $T^3 \times \zeta \RP^3$ and $\vartheta^{t}$ agrees with $\ast_{\GTwostd} \GTwostd$ near the boundary. 
Thus, the forms (\ref{eq: preperturbed GTwo form}) - (\ref{eq: preperturbed error}) can be extended by $\varphi^{\std}$ on the rest of $T^7/\Gamma \setminus U$.
Note that $\varphi^{t}$ is closed, while $\vartheta^{t}$ and $\psi^{t}$ are not.

By \cite[Theorem 7.1]{Joyce1999ExceptionalHolonomy}, see also \cite[Chapter 11.5]{Joyce2000SpecialHolonomy}, the forms $\varphi^t$ and $\psi^t$ have following properties:
\begin{theorem}\label{Theorem - Joyce's form estimation}
    There are positive constants $A_1,A_2,A_3,$ and $\varepsilon$ such that the one-parameter family of forms $\varphi^t$ and $\psi^t$ satisfy, for all $t \in (0,\varepsilon]$, the following identities:
    \begin{itemize}
        \item[(o)] $\diff^{\ast_{\varphi^t}} \psi^t = \diff^{\ast_{\varphi^t}} \varphi^t$.
        \item[(i)] $||\psi^t||_{L^2} \leq A_1t^4$, $||\psi^t||_{C^0} \leq A_1 t^3$ and $||\diff^{\ast_{\varphi^t}} \psi_t||_{L^{14}} \leq A_1 t^4$.
        \item[ii)] The injectivity radius $\mathrm{inj}(g_{\varphi^t})$ satisfies $\mathrm{inj}(g_{\varphi^t}) \geq A_2 t$.
        \item[(iii)] The Riemannian curvature tensor $R(g_{\varphi^t})$ satisfies $||R(g_{\varphi^t})|| \leq A_3 t^{-2}$.
    \end{itemize}
    Here, the norms $||\placeholder||_{C^0}$, $||\placeholder||_{L^2}$, and $||\placeholder||_{L^{14}}$ depend on the metric $g_{\varphi^t}$.
\end{theorem}
By \cite[Theorem 11.6.1]{Joyce2000SpecialHolonomy}, which in turn relies on Theorem G1 and Theorem G2 in \cite{Joyce2000SpecialHolonomy}, there is a torsion-free $\GTwo$-form $\tilde{\varphi}^t$ such that $||\Tilde{\varphi}^t - \varphi^t||_{C^0(\varphi^t)} = O(t^{1/2})$.
In particular, $\mathrm{Hol}(g_{\tilde{\varphi}}^t) \leq \GTwo$.

\vspace{5pt}

Let us now turn our attention to the topological properties of $M$.
\begin{lemma}\label{Lemma - Fundamental Group of Resoultion}
    The resolution map $q \colon M \rightarrow T^7/\Gamma$ induces an isomorphism on their fundamental groups.
\end{lemma}
\begin{proof}
    The map $\id \times \Phi_a \colon T^3 \times EH \rightarrow T^3 \times \Quat/\Z_2$ is $\tilde{F}_a$-$F_a$-equivariant and
    it induces an isomorphism between the corresponding fundamental groups.
    The groups $\tilde{F}_a$ and $F_a$ act freely and properly discontinuously, so that the fundamental groups of the quotients are given by a semidirect product of $\Z^3 \rtimes \tilde{F}_a$ and $\Z^3 \rtimes F_a$, respectively.
    The Five Lemma now implies that the descent of $\id \times \Phi_a$ to the quotients induces an isomorphism between their fundamental groups.
    An iterative application of Seifert--van Kampen to the pushout diagram (\ref{eq - M as pushout}) now implies that $\pi_1(q) \colon \pi_1(M) \rightarrow \pi_1(T^7/\Gamma)$ is an isomorphism.
\end{proof}
Thus, if $T^7/\Gamma$ has a finite fundamental group, its resolution $M$ has a finite fundamental group, too, which implies $\Hol(g_{\tilde{\varphi}^t}) = \GTwo$.

To prove Theorem \ref{Main Theorem - Example}, we require information about the divisibility of the integral Pontryagin class (or rather its image in the integral lattice).
To this end, recall that the universal coefficient theorem implies that the evaluation homomorphism $H^k(M;\R) \rightarrow \mathrm{Hom}_{\Z}(H_k(M;\Z),\R)$ is an isomorphism.
Using this isomorphism, we define the \emph{integral lattice} $H^k(M;\Z \subseteq \R)$ as the subgroup that corresponds to $\mathrm{Hom}_{\Z}(H_k(M;\Z),\Z)$ under the evaluation homomorphism.

\begin{theorem}\label{Theorem - DivPontrClass}
    If $p_1(M)$ is the first Pontryagin class of a generalised Kummer construction from above, then $p_1(M) \in H^4(M;3\Z \subseteq \R)$.
    In particular, for all smooth maps from each closed, oriented four manifold $f \colon N^4 \rightarrow M$, we have
    \begin{equation*}
        \langle p_1(M);f_\ast([N^4])\rangle = \int_{N^4} f^\ast p_1(M) \in 3\Z.
    \end{equation*}
\end{theorem}

The Pontryagin class inside the integral lattice of the real cohomology groups can be described using Chern-Weil theory.
If we use the pre-perturbed $\GTwo$-form on $M$, then the curvature of its induced metric vanishes in the "interior" $T^7/\Gamma \setminus \Bar{U} = M \setminus \bar{V}$ of the resolved manifold $M$, so that the Pontryagin $4$-form obtained by Chern-Weil theory is supported within $V$.
We can therefore deduce Theorem \ref{Theorem - DivPontrClass} from local considerations.

\begin{proof}
    The $\GTwo$-form $\varphi^t$ induces a product metric $g_{\varphi^t}=\langle \placeholder,\placeholder\rangle_{T^3} \oplus h^t$, where the first factor is the flat metric on $T^3$ and the second factor is a Riemannian metric on $EH$ that is flat outside of $EH_{\leq 0.9\, \zeta}$.
    This metric is $\tilde{F}_a$-invariant, so it descends to the quotient, too.
    
    As a consequence, the $4$-form obtained from the Chern-Weil construction applied to the Levi-Civita connection of $g_{\varphi^t}$ provides a lift $\tilde{p}_1(T^3 \times EH) \in H^4(T^3\times EH_{\leq \zeta},T^3\times \partial EH_{\leq \zeta})$.
    Since the metric is of product form, this lift satisfies $\tilde{p}_1(T^3 \times EH) = \mathrm{pr}_2^\ast \, \tilde{p}_1(EH)$, where $\tilde{p}_1(EH) \in H^4(EH_{\leq \zeta},\partial EH_{\leq \zeta})$ is defined analogously using the metric $h^t$.

    To determine $\tilde{p}_1(EH)$, we recall that the $K3$ surface can be obtained by resolving a flat torus $T^4/\Z_2$ along its singular set consisting of $16$ points, see, for example, \cite[Example 7.3.14]{Joyce2000SpecialHolonomy}.
    As the metric $h^t$ is flat near the boundary, it can be extended with the flat metric to the "interior" of the $K3$ surface.
    Thus, the differential form that represents $p_1(K3)$ obtained by applying Chern-Weil theory to that metric is again supported in the $16$ copies of $EH_{< \zeta}$.
    Hirzebruch's signature theorem now yields
    \begin{align*}
        -16 = \mathrm{sign}(K3) &= \frac{1}{3}\int_{K3} p_1(K3) = \frac{16}{3} \int_{EH_{\leq \zeta}} \Tilde{p}_1\left(EH_{\leq \zeta}\right),
    \end{align*}
    which implies that $\tilde{p}_1(EH) = - 3 \cdot \mathrm{gen}_{EH}$, where $\mathrm{gen}_{EH}\in H^4(EH_{\leq \zeta},\partial EH_{\leq \zeta}; \Z) \cong \Z$ is represented by any four form (supported in $EH_{<\zeta}$) that integrates to $1$.

    Next consider a smooth map $f\colon (N^4,\partial N^4) \rightarrow T^3 \times (EH_{\leq \zeta},\partial EH_{\leq \zeta})$, where $N^4$ a compact, oriented manifold with (possibly empty) boundary. 
    Since $\tilde{F}_a$ acts on $T^3 \times EH_{\leq \zeta}$ in an orientation preserving manner, it does so, too, on the total space of the pullback
    \begin{equation*}
        \xymatrix{ \hat{N}^4 \ar[d]_{q} \ar[rr]^-F && T^3 \times EH_{\leq \zeta} \ar[d]^{\kappa} \\ N^4 \ar[rr]^-f && \bigl(T^3 \times EH_{\leq \zeta}\bigr)/\tilde{F}_a,  }
    \end{equation*}
    which is again an compact, oriented manifolds with boundary.
    Thus, the fundamental classes of these two manifolds are related by $q_\ast [\hat{N}^4,\partial \hat{N}^4] = |\tilde{F}_a|\cdot [N^4,\partial N^4]$.

    Using naturality of the Chern-Weil construction with respect to local isometries, it follows that
    \begin{align*}
        \langle \tilde{p}_1((T^3\times EH)/\tilde{F}_a), f_\ast[N^4,\partial N^4] \rangle &= |\tilde{F}_a|^{-1}\langle \tilde{p}_1((T^3\times EH)/\tilde{F}_a), f_\ast q_\ast[\hat{N}^4,\partial \hat{N}^4] \rangle \\
        &= |\tilde{F}_a|^{-1}\langle \tilde{p}_1((T^3\times EH)/\tilde{F}_a), \kappa_\ast F_\ast[\hat{N}^4,\partial \hat{N}^4] \rangle \\
        &= |\tilde{F}_a|^{-1}\langle \tilde{p}_1(T^3\times EH),  F_\ast[\hat{N}^4,\partial \hat{N}^4] \rangle \\
        &= |\tilde{F}_a|^{-1}\langle \tilde{p}_1(EH), (\mathrm{pr}_2 \circ F)_\ast[\hat{N}^4,\partial \hat{N}^4] \rangle\\
        &= -3/|\tilde{F}_a| \langle \mathrm{gen}_{EH}, (\mathrm{pr}_2 \circ F)_\ast[\hat{N}^4,\partial \hat{N}^4] \rangle =: -3 / |\tilde{F}_a| \underbrace{\lambda_a}_{\in \Z}.
    \end{align*}

    Finally, each smooth map $f \colon N^4 \rightarrow M^7$ can be homotopied to a map that intersects the smooth submanifold $\partial \bar{V} \cong \bigsqcup_a (T^3 \times \RP^3)/\tilde{F}_a$ transversely for the space of these maps is an open and dense subspace of $C^\infty(N^4,M)$, see \cite[Theorem 3.2.1]{hirsch1997differential}.
    It follows that $f^{-1}(\bar{V}) \subseteq N^4$ is a manifold with boundary of codimension zero.
    From the previous discussion, we deduce
    \begin{align*}
        \langle p_1(M), f_\ast [N^4] \rangle &= \sum_{a \in \pi_0(S)} \langle \tilde{p}_1\bigl((T^3 \times EH)/\tilde{F}_a \bigr), f_\ast [N^4] \rangle \\
        \begin{split}&= \sum_{a \in \pi_0(S)} \langle \tilde{p}_1\bigl((T^3 \times EH)/\tilde{F}_a \bigr),\\
        & \qquad \qquad \quad f_\ast [f^{-1}((T^3 \times EH_{\leq \zeta})/\tilde{F}_a), f^{-1}((T^3 \times \partial EH_{\leq \zeta})/\tilde{F}_a) ] \rangle \end{split}\\
        &= -3 \sum_{a \in \pi_0(S)} \frac{1}{|\tilde{F}_a|} \lambda_a.
    \end{align*}
    
    Since $\langle p_1(M), f_\ast [N^4] \rangle \in \Z$ and since $|\tilde{F}_a|$ is prime to $3$,  we deduce that $\langle p_1(M), f_\ast [N^4] \rangle$ has to be divisible by $3$.

    By Thom's work \cite{Thom1954QuelquesPropertiesGlobal}, every class in $H_4(M;\Z)$ is realised by an embedded, closed, oriented submanifold of dimension $4$, hence $p_1(M) \in H^4(M,3\Z \subseteq \R)$, as claimed. 
\end{proof}

%% file: BundleResolution.tex
\section{Bundles of Joyce’s Resolutions}\label{Section - Bundle Resolution}

We use the results of Section \ref{Section - EH-Bundle} to carry out the construction of the previous section in families. 
We first begin with the construction of the fibre bundle $E_M \rightarrow \CP^1$ together with forms $\bm{\varphi}^t$ and $\bm{\psi}^t$ on its vertical tangent bundle, which are parametrised versions of (\ref{eq: preperturbed GTwo form}) and (\ref{eq: preperturbed error}).
We will then generalise Joyce's perturbation procedure, namely Theorem 11.6.1 in \cite{Joyce2000SpecialHolonomy}, to deform $\bm{\varphi}^t$ a fibrewise torsion-free $\GTwo$-form on $E_M$.

\subsection{Construction of the Bundles}
Let $M$ be a generalised Kummer construction from Section \ref{Section - GeneralisedKummer} and $\mathcal{EH} \rightarrow \CP^1$ be the Eguchi-Hanson bundle, which comes with a family of blow-up maps $\bm{\sigma} \colon \mathcal{EH} \rightarrow \Quat/\Z_2 \times \CP^1$. 
In complete analogy to the unparametrised case, we define $\mathcal{EH}_{< R}$ to be the preimage of the open balls $B_{R}^4(0)/\Z_2 \times \CP^1$ under $\bm{\sigma}$.

Using the notation of Condition \ref{Conditions - GroupsAction}, we define $\mathsf{A} \subseteq \pi_0(S)$ to be the subset of path components for which $F_a = \{1\}$.
For each subset $\mathsf{B} \subseteq \mathsf{A}$, we define a bundle $E_{M,\mathsf{B}}$ as the following pushout:
\begin{equation}\label{eq - bundle as pushout}
    \xymatrix@C+1em{ \bigcup_{b \in \mathsf{B}}\left( T^3 \times \zeta\RP^3 \right) \times \CP^1 \ar@{^{(}->}[d] \ar[rr]^-{\id \times  \bm{\sigma}^{-1} } && \bigcup_{b \in \mathsf{B}} T^3 \times \mathcal{EH}_{\leq \zeta} \ar[d] \\
    \left( M \setminus \bigcup_{b \in \mathsf{B}} T^3 \times EH_{<\zeta} \right) \times \CP^1 \ar[rr] && E_{M, \mathsf{B}.}}
\end{equation}
By the pushout property, the projections $\left( M \setminus \bigcup_{b \in \mathsf{B}} T^3 \times EH_{<\zeta} \right) \times \CP^1 \rightarrow \CP^1$ and $T^3 \times \mathcal{EH}_{\leq \zeta} \rightarrow \CP^1$ yield a map $p \colon E_{M,\mathsf{B}} \rightarrow \CP^1$.

The set $\mathsf{B}$ parametrises the singularities of the orbifold that we resolve with the twisted family $\mathcal{EH}$. 
The remaining path components are resolved with the untwisted family $EH \times \CP^1$.

\begin{lemma}
  The map $p \colon E_{M,\mathsf{B}} \rightarrow \CP^1$ turns $E_{M,\mathsf{B}}$ into a smooth fibre bundle over $\CP^1$ with fibre $M$.
\end{lemma}
\begin{proof}
    To avoid notational clutter, we assume that $\mathsf{B}$ consists of a single element, which we drop from the notation. 
    Since all constructions are carried out in a neighborhood of the singularities, this assumption is without restriction.
    
    We first provide convenient local trivialisations of $\mathcal{EH}$ over the open subsets $U_j := \{ [z_0:z_1] \, : \, z_j \neq 0 \} \subseteq \CP^1$.
    On these subsets, we define $s_j \colon U_j \rightarrow S^3$ via 
    \begin{equation*}
        s_0 \colon [z_0:z_1] \mapsto \frac{1 + z_1/z_0 \cdot \ju}{|1 + z_1/z_0 \cdot \ju|} \qquad \text{ and } \qquad s_1 \colon [z_0:z_1] \mapsto \frac{z_0/z_1 + \ju}{|z_0/z_1 + \ju|}.
    \end{equation*}
    The local trivialisations are now given by:
    \begin{equation*}
       \Xi_j \colon \mathcal{EH}|_{U_j} \rightarrow EH \times U_j \qquad \text{via} \qquad \Xi_j([p,[q],\lambda]) = ([s_j(q)p,\lambda],[q]).  
    \end{equation*}
    Since $U_j$ is contractible, we find homotopies $H_j \colon U_j \times [1,2] \rightarrow S^3$ such that $H_j(\placeholder, 1) = s_j$ and $H_j(\placeholder,2) = 1$.
    If we choose $\zeta$ in the construction of $M$ sufficiently small, we may assume that the canonical embedding of $T^3 \times EH_{\leq \zeta}$ into $M$ extends to $T^3 \times EH_{\leq 2\zeta}$.
    
    We now define the maps $\Theta_j \colon \left( M \setminus T^3 \times EH_{<\zeta} \right) \times U_j \rightarrow \left( M \setminus T^3 \times EH_{<\zeta} \right) \times U_j$ via
    \begin{equation*}
        \Theta_j \colon (y,[q]) \mapsto \begin{cases}
            (\tau,[H_j([q],|\lambda|/\zeta)p,\lambda],[q]), & \text{if } y = (\tau,[p,\lambda]) \in T^3 \times \left(EH_{\leq 2\zeta} \setminus EH_{< \zeta} \right), \\
            (y,[q]), & \text{if } y \notin M \setminus T^3 \times EH_{\leq 2\zeta}.
        \end{cases} 
    \end{equation*}
    By the construction of $\Theta_j$ the following diagram commutes:
    \begin{equation*}
        \xymatrix@C+1em{ \left( M \setminus T^3 \times EH_{<\zeta} \right) \times U_j \ar[d]^{\Theta_j} & \ar@{_{(}->}[l] T^3 \times \zeta\RP^3 \times U_j \ar[rr]^-{\id \times  \bm{\sigma}^{-1}} \ar[d]^{(\tau,[x],[q]) \mapsto (\tau,[s_j([q])x],[q])} && T^3 \times \mathcal{EH}_{\leq \zeta}|_{U_j} \ar[d]^{\Xi_j} \\
        \left( M \setminus T^3 \times EH_{<\zeta} \right) \times U_j  & \ar@{_{(}->}[l] T^3\times \zeta\RP^3 \times U_j \ar[rr]^-{\id \times  {\sigma}^{-1} } && T^3 \times \EH_{\leq \zeta}|_{U_j},}
    \end{equation*}
    so the pushout property yields a local trivialisations $\Psi_j \colon E_{M}|_{U_j} \rightarrow M \times U_j$.
    The trivialisations are smooth if the homotopies $H_j$ are smooth and stationary near the boundaries.
\end{proof}

The definition of $\Check{r}$ from the previous section readily adapts to the parametrised case by using $\bm{\sigma}$ instead of $\sigma = \bm{\sigma}_1$, so we get a map $\check{\mathbf{r}} \colon E_{M}\rightarrow \R_{\geq 0}$.
We define parametrised versions of (\ref{eq: preperturbed GTwo form}) - (\ref{eq: preperturbed error}), using once more the index set $\{\iu,\ju,\ku\}$ instead of $\{1,2,3\}$, as follows:
\begin{align}
    \bm{\varphi}^t_{[q]} &:= \ast_{\varphi^{\mathcal{EH},t}_{[q]}} \varphi^{\mathcal{EH},t}_{[q]} - \diff \bigl(\chi(\check{\mathbf{r}}^2) \tau_{q^{-1}\iu q}^t\bigr), \label{eq: family preperturbed GTwo form} \\
    \bm{\vartheta}^t_{[q]} &:= \ast_{\varphi^{\mathcal{EH},t}_{[q]}}\varphi^{\mathcal{EH},t}_{[q]} -(\omega_{\iu}^{\EH, q^{-1}\iu q}) \wedge \diff (\chi(\Check{\mathbf{r}}^2)\tau_{q^{-1}\iu q}^t) + \frac{1}{2} \diff (\chi(\Check{\mathbf{r}}^2)\tau_{q^{-1}\iu q}^t) \wedge \diff (\chi(\Check{\mathbf{r}}^2)\tau_{q^{-1}\iu q}^t) \label{eq: family preperturbed first integral} \\
    & \hspace{2.14cm} - \diff (\chi(\Check{\mathbf{r}}^2)\tau_{q^{-1}\iu q}^t) \wedge \delta^{q^{-1}\ju q} \wedge \delta^{q^{-1}\ku q}, \nonumber \\
    \bm{\psi}^t_{[q]} &:= \bm{\varphi}^t_{[q]} - \ast_{\bm{\varphi}^t_{[q]}} \bm{\vartheta}^t_{[q]}. \label{eq: family preperturbed error}
\end{align}
The form $\bm{\vartheta}^{t}_{[q]}$ is well defined, because an explicit calculation shows that $\delta^{q^{-1}\ju q} \wedge \delta^{q^{-1}\ku q}$ does not depend on the choice of the representative $q$ of $[q] \in \CP^1$.
%
%

These three forms further agree with their stationary counterparts (\ref{eq: preperturbed GTwo form}) - (\ref{eq: preperturbed error}) in a neighbourhood of $\bigcup_{b \in \mathsf{B}}\left( T^3 \times \zeta\RP^3 \right)$, so they extend to well defined forms on the entire bundle $E_{M,\mathsf{B}}$. 
Furthermore, $\bm{\varphi}^t_{[q]}$ and $\bm{\psi}_{[q]}^t$ satisfy the assumptions of Theorem \ref{Theorem - Joyce's form estimation} because they were constructed as their counterparts in Section \ref{Section - GeneralisedKummer}, we just use different admissible choices when we vary the base parameter $[q]$.

\subsection{A Perturbation Result for Families}
We would like to apply a family version of \cite[Theorem 11.6.1]{Joyce2000SpecialHolonomy} in the sense that if we start with a family of closed $\GTwo$-forms $\bm{\varphi}_{[q]}^t$ that depends continuously on the parameter, then we get a family of torsion-free $\GTwo$-forms $\Tilde{\bm{\varphi}}^t_{[q]}$ that also depends continuously on the parameter.

More abstractly, assume that we have a fibre bundle $E_M \rightarrow B$ over a compact base space $B$ that carries a fibrewise closed $\GTwo$-structure $\{\varphi_b\}_{b \in B}$, that is, an element $\varphi \in \Omega^{3,\vertical}(E) = \Gamma(E,\Lambda^3T^{\vertical,\vee}E)$. 
Of course, this yields a family of fibre-metrics $\{g_{\varphi_b}\}_{b\in B}$.

The family version of \cite[Theorem 11.6.1]{Joyce2000SpecialHolonomy} reads as follows:
\begin{theorem}\label{Theorem - Family Joyce Perturbation}
    Let $A_1$, $A_2$, $A_3$ be positive constants. 
    Then there exists positive constants $\kappa$, $K$ such that, whenever $0 < t \leq \kappa$, the following is true:

    Let $E_M \rightarrow B$ be a smooth fibre bundle over a compact base space $B$ and let $\{\varphi\}_{b \in B}$ be a fibrewise $\GTwo$-form on $E_M$ with $\diff \varphi_b = 0$ for all $b\in B$.
    Suppose that $\psi = \{\psi_b\}_{b \in B}$ is a fibrewise smooth $3$-form on $E_M$ with $\diff^{\ast_{\varphi_b}} \psi_b = \diff^{\ast_{\varphi_b}} \varphi_b$, and
    \begin{itemize}
        \item[(i)] $||\psi_b||_{L^2} \leq A_1 t^4$, $||\psi_b||_{C^0} \leq A_1t^{1/2}$, and $||\diff^{\ast_{\varphi_b}}\psi_b||\leq \lambda$,
        \item[(ii)] the injectivity radius $\delta(g_{\varphi_b})$ satisfies $\delta(g_{\varphi_b})\geq A_2 t$, and
        \item[(iii)] the Riemannian curvature tensor $R(g_{\varphi_b})$ satisfies $||R(g_{\varphi_b})||_{C^0} \leq A_3 t^{-2}$,
    \end{itemize}
    where the norms $||\placeholder||_{L^2}$, $||\placeholder||_{C^0}$, and $||\placeholder||_{L^{14}}$ arise from $g_{\varphi_b}$.
    Then there is a fibrewise $\GTwo$-form $\tilde{\varphi} = \{\tilde{\varphi}\}_{b \in B} $ that depends continuously on the parameter and that satisfies $||\tilde{\varphi}_b - \varphi_b||_{C^0} \leq K t^{1/2}$.  
\end{theorem}
The proof of this theorem is little original. 
Once we know that the two major ingredients of the proof of \cite[Theorem 11.6.1]{Joyce2000SpecialHolonomy}, namely Theorem G1 and G2 in loc. cit., generalises to the family set up, then the proof of Joyce straightforwardly generalises, too.

Theorem G1 in \cite{Joyce2000SpecialHolonomy} generalises by its very formulation to the family setting, because one simply has to replace the injectivity radius and the curvature bound of a single metric by the minimum of all injectivity radii and the maximum of all curvatures over the parameter space, respectively. 

Theorem G2 in \cite{Joyce2000SpecialHolonomy} generalises to the family setting as follows:
\begin{theorem}[Theorem G2]
    Let $\lambda, C_1$, and $C_2$ be positive constants. 
    Then there exist positive constants $\kappa$, $K$ such that whenever $0 < t \leq \kappa$, the following is true:

    Let $E \rightarrow B$ a fibre bundle over a compact base space $B$ with smooth, closed fibre $M$ and a fibrewise closed $\GTwo$-structure on it.
    Suppose that $\psi = \{\psi_b\}_{b\in B} \in \Omega^{3,\vertical}(E) = \Gamma(E,\Lambda^3T^{\vertical,\vee}E)$ fibrewise $3$-form with $\diff^{\ast_{\varphi}} = \diff^{\ast_\varphi} \varphi$ and
    \begin{itemize}
        \item[(i)] $||\psi_b||_{L^2} \leq \lambda t^4$, $||\psi_b||_{C^0} \leq \lambda t^{1/2}$, and $||\diff^{\ast_{\varphi_b}} \psi_b||_{L^{14}} \leq \lambda$,
        \item[(ii)] if $\chi \in L_1^{14}(\Lambda^3T^\vee E_{b})$, then $||\nabla^{\varphi_b}\chi|| \leq C_1( ||\diff \chi||_{L^{14}} + ||\diff^{\ast} \chi||_{L^{14}} + t^{-4}||\chi||_{L^2} ),$\footnote{That means that $(ii)$ in Joyces theorem holds true for all fibres. We require the same for $(iii)$.}
        \item[(iii)] if $\chi \in L_1^{14}(\Lambda^3 T^\vee E_b)$, then $||\chi||_{C^0} \leq C_2(t^{1/2} ||\nabla^{\varphi_b} \chi||_{L^{14}} + t^{-7/2} ||\chi||_{L^2} )$.
    \end{itemize}
    Let $\varepsilon_1$ be the universal constant from \cite[Definition 10.3.3]{Joyce2000SpecialHolonomy} and $F$ as in \cite[Proposition 10.3.5]{Joyce2000SpecialHolonomy}.
    Then there exists $\eta \in \Gamma^0(E,\Lambda^{2}T^\vertical E)$ and a function $f \in C^0(E)$, both smooth in fibre direction, such that 
    \begin{equation}\label{eq: Solution equation for G2 form graph}
        \Delta_\varphi \eta = \diff^{\ast_\varphi} \psi + \diff^{\ast_\varphi} (f\psi) + \ast \diff F(\diff \eta) \qquad \text{and} \qquad f\varphi = \frac{7}{3} \pi_1(\diff \eta).
    \end{equation}
\end{theorem}
\begin{proof}
    The existence of such a section $\eta$ and a function $f$ follows from a pointwise application of the proof of Theorem G2 in Joyce, which relies on a Banach fixpoint iteration of equation (\ref{eq: Solution equation for G2 form graph}) as described in the proof of \cite[Proposition 11.8.1]{Joyce2000SpecialHolonomy}.
    We claim that if we use in each iteration step the unique solution that is perpendicular to the kernel $\Delta_\varphi$, then the resulting section $\eta$ will depend continuously on the base parameter.
    Since continuity is a local property, we may assume that the bundle is trivial, so that $\varphi$, $\psi$, and all $\eta$'s will be maps.

    The iteration step begins with $\eta_0 = 0$ and $f_0 = 0$, which are certainly continuous sections.
    Thus, it suffices to prove that if $\eta_j$ and $f_j$ are continuous in the base parameter $b$, then $\eta_{j+1}$ is so, too, and hence $f_{j+1}$ as well.

    By construction, the forms $\diff^{\ast_\varphi}\psi + \diff^{\ast_\varphi}(f_k\psi) + \ast \diff F(\diff \eta_j)$ lie in the image of $\diff^{\ast_\varphi}$, 
    which is perpendicular to the harmonic forms by Hodge theory.
    Theorem \ref{Theorem - Continuity Laplace Operators} below now implies that  $\eta_{j+1}$ is continuous because the restriction of the invertible operator $\mathrm{pr}_{\ker \Delta_\varphi} + \Delta_\varphi$, which depends continuously on $\varphi$, to the orthogonal complement of $\ker \Delta_\varphi$ (with respect to the inner product induced by $g_\varphi$) is just $\Delta_\varphi$.
\end{proof}

We will need to study the dependency of solutions of Hodge-Laplace equations under the underlying metric.
To this end, recall from \cite[Observation 3.6]{ebert2017indexdiff} that, for any two Riemannian metrics $g_1, g_2 \in \Riem(M)$, there is a vector bundle homomorphism\footnote{In notation of \cite{ebert2017indexdiff}, we would write $\mathcal{G}_{g_1,g_2} = T_{(g_1,\Lambda^{3,\vee}g_1),(g_2,\Lambda^{3,\vee}g_2)}$.} $\mathcal{G}_{g_1,g_2} \in \mathrm{End}(\Lambda^3 T^\vee M)$ that induces a unitary isomorphism $L^2(\Lambda^3 T^\vee M;g_1) \rightarrow L^2(\Lambda^3 T^\vee M, g_2)$, where these Hilbert spaces carry the canonical $L^2$-inner products arising from the Riemannian metrics.
Moreover, for each metric $g_0$, one can construct these homomorphisms such that the map $\mathcal{G}_{\placeholder,g_0}\colon \Riem(M) \rightarrow \mathrm{End}(\Lambda^3T^\vee M)$ is smooth.

The first ingredient is a technical lemma that the kernel of a self-adjoint operator depends continuously on the operator as long as the dimension does not change.
The proof is quite similar to the proof of \cite[Lemma III.7.2]{LawsonMichelsonSpin}.
\begin{lemma}\label{Lemma - Kernel Bundle}
    Let $H$ be a Hilbert space and denote by $\mathrm{Fred}(H)^{\mathrm{s.a}}$ the space of all self-adjoint Fredholm operators on $H$ equipped with the operator norm topology.
    Define $F_j \subseteq \mathrm{Fred}^{\mathrm{s.a}}(H)$ to be the closed subspace of all operators $A$ satisfying $\dim \ker A \geq j$.

    Then the map induced by the projection to the first component yields a vector bundle 
    \begin{equation*}
        \mathrm{ker}_j := \{ (T,u) \, : \, Tu = 0 \} \subseteq F_j\setminus F_{j+1} \times H \xrightarrow{\mathrm{pr}_1} F_j \setminus F_{j+1}.
    \end{equation*}
    In particular, the map $F_j \setminus F_{j+1} \rightarrow B(H)$ that assigns to an operator the orthogonal projection onto its kernel is continuous.
\end{lemma}
\begin{proof}
    Pick $T_0 \in F_j \setminus F_{j+1}$ and define, for each $T \in F_j \setminus F_{j+1}$, the map
    \begin{equation*}
        E_{T_0}(T) \colon \ker T_0 \oplus \ker T_0^\perp \rightarrow H \qquad \qquad (w,v) \mapsto w + T(w) + T(v).
    \end{equation*}
    For each $T_0$, the map $B(H)^{\mathrm{s.a}} \rightarrow B(\ker T_0 \oplus \ker T_0^\perp, H)$ given by $T_0 \mapsto E_{T_0}(T)$ is continuous with respect to the norm topologies, and, furthermore, $E_{T_0}(T_0)$ is an isomorphism.

    Since invertibility is an open condition, we find an open neighbourhood $\mathcal{U}$ of $T_0$ such that the restriction of $E_{T_0}$ to $\mathcal{U}$ takes values in the subspace of all isomorphisms.
    By construction, we have $E_{T_0}(T)(\ker T) \subseteq \ker T_0$.
    For degree reasons, we even have $E_{T_0}(T)(\ker T) = \ker T_0$.
    The local trivialisation around $T_0$ is now given by
    \begin{equation*}
        \mathrm{pr}_1 \times {E_{T_0}}: \mathrm{ker}_j \rightarrow \mathcal{U} \times \ker T_0,
    \end{equation*}
    so $\mathrm{ker}_j$ is vector bundle of rank $j$.

    For the second statement, we use that $\ker_j$ is locally trivial to pick a local orthonormal frame $(f_1,\dots,f_j)\colon \mathcal{U} \rightarrow \ker_j$.
    The map that assigns to $T \in F_j \setminus F_{j+1}$ the orthogonal projection onto its kernel can be written on $\mathcal{U}$ as 
    \begin{equation*}
        U \mapsto \sum_{\alpha = 1}^j \langle f_\alpha(U) , \placeholder \rangle f_\alpha(U),
    \end{equation*}
    which is clearly continuous.
\end{proof}

We denote by $L^2_k(\Lambda ^3 T^\vee M)$ the Sobolev space of regularity $k$ (and H\"older power $2$).
Each Riemannian metric $g$ induces a canonical Hilbert space structure on $L_k^2(\Lambda^3T^\vee M)$.
While the specific inner product depends crucially on the chosen Riemannian metric $g$, the underlying topological vector space does not because $M$ is closed.
In other words, two different choices of Riemannian metrics, give rise to equivalent norms, so we abbreviate the underlying topological vector space simply with $L_k^2$.

\begin{theorem}\label{Theorem - Continuity Laplace Operators}
    For each $k \in \R$, the following map is well defined and continuous: 
    \begin{align*}
        \Riem(M) &\rightarrow \mathrm{Isom}(L_k^2,L_{k-2}^2) \subseteq \mathrm{Hom}(L_k^2,L_{k-2}^2),\\
        g &\mapsto \mathrm{pr}_{\ker \Delta_g} + \Delta_g.  
    \end{align*}
    Here, $\mathrm{pr}_{\ker \Delta_g}$ denotes the orthogonal projection to the kernel with respect to the $L_0^2 = L^2$-inner product induced by $g$.
\end{theorem}
\begin{proof}
    We start with the proof that the map $\Riem(M) \rightarrow \mathrm{Hom}(L_k^2,L_{k-2}^2)$ is continuous.
    Expressing the Hodge-Laplace operator on $3$-forms in local coordinates shows that the map $\Riem(M) \rightarrow \Psi\mathrm{DO}^2(\Lambda^3T^\vee M)$ is continuous, where the target carries the Fr\'echet structure of \cite[p. 123f]{atiyah1971indexIV}. 
    Furthermore, the realisation map $\Psi\mathrm{DO}^{n}(\Lambda^3 T^\vee M) \rightarrow \mathrm{Hom}(L_k^2,L_{k-n}^2)$ that interprets a pseudo differential operator as a bounded operator between the corresponding Sobolev spaces is continuous for all $k \in \R$ and $n \in \Z$ if the domain carries the Fr\'echet topology and the target the norm topology.
    It follows that the map $\Delta \colon \Riem(M) \rightarrow \mathrm{Hom}(L_k^2,L_{k-2}^2)$ is continuous. 
   
    The map $g \mapsto \mathcal{G}_{g,g_0} \circ \Delta_g \circ \mathcal{G}_{g,g_0}^{-1}$ takes values in unbounded operators on $L^2(\Lambda^3 T^\vee M, g_0)$ that are self-adjoint with respect to $\langle \placeholder, \placeholder \rangle_{g_0}$ and whose domain is $L_2^2$.
    By \cite[Proposition 1.7]{Nicolaescu2007Fredholm} the map
    \begin{equation*}
       \Riem(M) \rightarrow B(L^2(M,g_0)) \qquad \text{given by} \qquad   g \mapsto \frac{\mathcal{G}_{g,g_0} \circ \Delta_g \circ \mathcal{G}_{g,g_0}^{-1}} { 1 + \mathcal{G}_{g,g_0} \circ \Delta_g \circ \mathcal{G}_{g,g_0}^{-1} } =: B_g
    \end{equation*}
    is continuous.

    Each $B_g$ is a bounded, self-adjoint  operator that has the same kernel as $\mathcal{G}_{g,g_0} \circ \Delta_g \circ \mathcal{G}_{g,g_0}^{-1}$.
    By Hodge theory, the dimension of  $\ker B_g \cong \ker \Delta_g \cong H^3(M)$ is independent of the underlying Riemannian metric.
    Lemma \ref{Lemma - Kernel Bundle} implies that the map $g \mapsto \mathrm{pr}_{\ker B_g} = \mathcal{G}_{g,g_0} \circ \mathrm{pr}_{\ker \Delta_g} \circ \mathcal{G}_{g,g_0}^{-1}$ is continuous, so the assignment $g \mapsto \mathrm{pr}_{\ker \Delta_g}$ yields a continuous map $\mathrm{Riem}(M) \rightarrow \mathrm{Hom}(L_0^2,L_0^2)$.

    We will prove inductively that $\pr_{\ker \Delta_{(\placeholder)}} \colon \Riem(M) \rightarrow \mathrm{Hom}(L_k^2,L_k^2)$ is continuous by showing that it is continuous at each point $g_0 \in \Riem(M)$ using a boot-strapping argument.
    To this end, first observe that elliptic regularity implies that $\ker \Delta_g$ consists of smooth sections only, so that $\pr_{\ker \Delta_g}$ is an infinite smoothing operator.

    Pick another metric $g_1 \in \Riem(M)$ and abbreviate $\pr_{\ker \Delta_{g_j}}$ to $\pr_j$.
    G\aa rdings inequality together with some triangle inequalities yield
    \begin{align*}
        \begin{split}
            &\qquad \qquad \ ||\pr_{1}(u) - \pr_{0}(u) ||_{k+2}\\
            &\leq C(g_0) \Bigl( ||(\Delta_{g_1} - \Delta_{g_0})(\pr_1(u) - \pr_0(u))||_{k} \\
            &\qquad \quad  + ||(\Delta_{g_1} - \Delta_{g_0}) \pr_0(u)||_k + ||\pr_1(u) - \pr_0(u)||_k   \Bigr)
        \end{split} \\
        &\leq C(g_0) \Bigl( \varepsilon(g_1,g_0)||\pr_1(u) - \pr_0(u)||_{k+2} + \varepsilon(g_1,g_0)||\pr_0(u)||_{k+2} + \tilde{\varepsilon}(g_1,g_0)||u||_k \Bigr).
    \end{align*}
    Since $\Delta_{(\placeholder)} \colon \Riem(M) \rightarrow \mathrm{Hom}(L_{k+2}^2, L_{k}^2)$ and $\pr_{\ker \Delta_{(\placeholder)}} \colon \Riem(M)\rightarrow \mathrm{End}(L_k^2)$ are continuous, we can find a small neighbourhood of $g_0$ to make $\tilde{\varepsilon}(g_1,g_0)$ and $\varepsilon(g_1,g_0)$ as small as we want so that the previous inequality can be rewritten as
    \begin{align}\label{eq: Continuity projection to kernel}
        ||\pr_{1}(u) - \pr_{0}(u) ||_{k+2} \leq C(g_0) \frac{\varepsilon(g_1,g_0)||\pr_0(u)||_{k+2} + \tilde{\varepsilon}(g_1,g_0) ||u||_k}{1 - C(g_0)\cdot \varepsilon(g_1,g_0)},
    \end{align}
    from which we can deduce that $g \mapsto \pr_{\ker \Delta_g}$ is continuous at $g_0$.

   A priori, the above argument would only prove the statement for the case that $k$ is a non-negative even integer, but if $k \in (-2,0)$, we can use $||\pr_1(u) - \pr_0(u)||_{k} \leq ||\pr_1(u) - \pr_0(u)||_0$, so that we may replace $||\placeholder||_k$ by $||\placeholder||_0$ in (\ref{eq: Continuity projection to kernel}).
   Hence we get continuity of $\pr_{\ker \Delta_{(\placeholder)}}$ in the range $(0,2)$ and by induction for all non-negative real numbers.
   
   The statement for negative $k$ now follows from duality $L_k^2 \cong (L_{-k}^2)'$.
\end{proof}

Applying Theorem \ref{Theorem - Family Joyce Perturbation} to the forms $\bm{\varphi}^t$ and $\bm{\psi}^t$ we deduce that $E_M$ has a section of fibrewise torsion-free $\GTwo$-forms.
\begin{theorem}\label{Theorem - Fibre torsion-free on Joyce Bundle}
    There are constants $\varepsilon, K > 0$  such that, for all $t < \varepsilon$ and all $\bm{\varphi}^t$ 
    from (\ref{eq: family preperturbed GTwo form}), there exists $\tilde{\bm{\varphi}}^t \in \Gamma^0(E_M,\Lambda^{3,+} T^{\vertical,\vee}E_M)$ that is 
    \begin{itemize}
        \item smooth in fibre direction, in other words, each $\tilde{\bm{\varphi}}^t_{[q]}$ is smooth,
        \item each $\tilde{\bm{\varphi}}^t_{[q]}$ is torsion-free, and
        \item satisfies $||\tilde{\bm{\varphi}}^t - {\bm{\varphi}}^t||_{C^0} \leq K t^{1/2}$.
    \end{itemize}
    Moreover, the forms $\Tilde{\bm{\varphi}}^t$ can be chosen such that they depend continuously on $t$.
\end{theorem}

%% file: ResolutionDetection.tex
\section{Topological Properties of the Bundles}\label{Section - Resolution Detection}

We will now show that the bundles $E_{M,\mathsf{B}} \rightarrow \CP^1$ constructed in the previous section are pair-wise different and then derive consequences for the moduli spaces of torsion-free $\GTwo$-structures from that. 
To this end, for each component of the singularity set of $T^7/\Gamma$ that belongs to $\mathsf{A}$, the set that parametrises all components $a \in \pi_0(S)$ of the fixpoint set with $F_a = \{1\}$, fix a point $b \in T^3$ inside that component. 
These singularities have a local neighbourhood isometrically isomorphic to $T^3 \times B_\zeta^4(0)/\Z_2 \subseteq T^3 \times \Quat /\Z_2$.
For the components that belong to the subset $\mathsf{B}$, the subset that parametrises those singularities that are resolved with the twisted Eguchi-Hanson family, we consider the following embeddings induced by the fibrewise zero sections
\begin{equation}\label{eq - inclusion twisted core spheres}
   \iota_b \colon \CP^1 \times \CP^1 \hookrightarrow \mathcal{EH}_{<\zeta} \hookrightarrow T^3 \times \mathcal{EH}_{< \zeta} \hookrightarrow E_{M,\mathsf{B}}, \qquad  ([p],[q]) \mapsto (b, [q^{-1}p,[q],0]).
\end{equation}
If the component does not belong to $\mathsf{B}$, we use the trivial Eguchi-Hanson bundle $EH \times \CP^1$ instead, in which case the map is then given by
\begin{equation}\label{eq - inclusion untwisted core spheres}
    \iota_b \colon \CP^1 \times \CP^1 \hookrightarrow EH_{<\zeta} \times \CP^1 \hookrightarrow T^3 \times EH_{<\zeta} \times \CP^1 \hookrightarrow E_{M,\mathsf{B}}, \   \quad ([p],[q]) \mapsto (b,[p,0],[q]).
\end{equation}
These four dimensional submanifolds will be used in an essential way in the proof of our key results.

\begin{lemma}\label{Lemma - Integral along twisted core spheres}
    The fibre-wise zero sections $\iota_b \colon \CP^1 \times \CP^1 \hookrightarrow E_{M,\mathsf{B}}$ satisfy
    \begin{equation*}
        p_1\bigl( T^{\vertical}E_{M,\mathsf{B}} \bigr) \cap \iota_{b\, \ast} [\CP^1 \times \CP^1] = \begin{cases}
            -8, & \text{ if } b \in \mathsf{B}, \\
            0, & \text{ if } b \notin \mathsf{B}.
        \end{cases}
    \end{equation*}
\end{lemma}
\begin{proof}
    Since $T^3 \times \mathcal{EH}_{<\zeta}$ is an open subset of $E_{M,\mathsf{B}}$ that contains $\iota_b(\CP^1 \times \CP^1)$ if $b \in \mathsf{B}$, Lemma \ref{Lemma - Char Clases Eguchi Hanson family} implies that the integral of the first Pontryagin class of the vertical tangent bundle $p_1(T^\vertical E_{M,\mathsf{B}})$ over $\CP^1 \times \CP^1$ is given by
    \begin{align*}
        \langle p_1(T^\vertical E_{M,\mathsf{B}});\iota_{b\, \ast}[\CP^1\times \CP^1] \rangle &= \langle p_1(T^\vertical E_{M,\mathsf{B}}|_{\CP^1 \times \CP^1}); [\CP^1 \times \CP^1]) \\
        &= \langle p_1(\underline{\R}^3 \oplus T^\vertical \mathcal{EH}|_{\CP^1 \times \CP^1}); [\CP^1 \times \CP^1] \rangle \\
        &= \langle -8xy;[\CP^1 \times \CP^1]\rangle = -8.
    \end{align*}

    If $b \notin \mathsf{B}$, then $\iota_b(\CP^1 \times \CP^1)$ is contained in the open neighbourhood $T^3 \times EH \times \CP^1 \subseteq E_{M,\mathsf{B}}$. 
    Hence, the restriction of $T^\vertical E_{M,\mathsf{B}}$ to this neighbourhood is isomorphic to the tangent bundle of the fibre $T^3 \times EH$, which has a trivial Pontryagin class. 
    A similar calculation as before now implies the second result.
\end{proof}


\begin{lemma}\label{Lemma - Fibre Inclusion Third Homology}
    Let $E_{M,\mathsf{B}} = E_M$ be a bundle of generalised Kummer constructions as defined in diagram \ref{eq - bundle as pushout}.
    Then each fibre inclusion $M \hookrightarrow E_M$ induces a monomorphism between their third integral homology groups.
    Furthermore, the images $H_4(\iota_a)([\CP^1 \times \CP^1])$ for all $a \in \mathsf{A}$ form a free subgroup in $H_4(E_M,\Z)$.
\end{lemma}
\begin{proof}
    Because we will swap the order between the fibre $\CP^1$ and the base $\CP^1 $ in this proof, we will denote the latter by $S^2$ for the sake of clarity.
    
    Let $T^7/\Gamma$ be a flat orbifold from which $M$ is obtained by resolving all singularities.
    Let $N$ be the orbifold that is obtained by resolving only those components of the singular set that do \emph{not} have a tubular neighbourhood of the form $T^3 \times B_\zeta^4(0)/\Z_2$, in other words, those components that do not belong to $\mathsf{A}$.
    Note that $M$ is obtained from $N$ by resolving the remaining singularities in the usual manner and that $\pi_1(M) \cong \pi_1(N)$ by an iterated application of Seifert--van-Kampen.
    Let further $\hat{N}$ be the manifold with boundary that is obtained from $N$ by cutting out tubular neighbourhoods of the singular set components that belong to $\mathsf{A}$.
    By construction, we get a  pushout as in \ref{eq - bundle as pushout}.

    The resolving map $\bm{\sigma} \colon \mathcal{EH} \rightarrow \Quat/\Z_2 \times S^2$, $\sigma \colon EH \rightarrow \Quat/\Z_2$, and the identity of $\hat{N}$ produce a fibre map $q \colon E_M \rightarrow N \times S^2$.   
    This map of pushout diagrams induces maps between the corresponding Mayer-Vietoris sequences of integral homology groups, in which we refer to the trivial fibre bundles over $S^2$ by underlining the fibres:
    \begin{equation*}
        \xymatrix{ H_k(\underline{T^3 \times \RP^3})^{{|\mathsf{A}|}}  \ar@{=}[d] \ar[r]^-{H_k(\iota_1)}_-{H_k(\iota_2)} & H_k(\underline{\hat{N}}) \oplus \Bigl(H_k(T^3 \times \mathcal{EH})^{\oplus |\mathsf{B}|} \oplus H_k(\underline{T^3\times EH})^{\oplus |\mathsf{A} - \mathsf{B}|} \Bigr) \ar[r] \ar@<-8.5ex>[d]^{H_k(\id \times \bm{\sigma})} \ar@<12.5ex>[d]^{H_k(\id \times \sigma \times \id)} \ar@{=}@<-24ex>[d] & H_k({E_M})  \ar[d]^{H_k(q)} \\
         H_k(\underline{T^3 \times \RP^3})^{{|\mathsf{A}|} } \ar[r]^-{H_k(\iota_1)}_-{H_k(\iota_2)} & H_k(\underline{\hat{N}}) \oplus \Bigl( H_k(\underline{T^3 \times \Quat/\Z_2})^{|\mathsf{B}|} \oplus H_k(\underline{T^3 \times \Quat/\Z_2})^{\oplus |\mathsf{A}-\mathsf{B}|} \Bigr)\ar[r] & H_k(\underline{N}).}
    \end{equation*}

   The Künneth formula implies that, for each component, the kernel of $H_k(\id \times \bm{\sigma})$ or $H_k(\id \times \sigma \times \id)$ is isomorphic to $H_{k-2}(T^3 \times S^2) \otimes H_2(\CP^1)$.
    
   The Künneth formula further implies that the image of $H_k(\iota_2)$ intersects the kernel of $H_k(\id \times \bm{\sigma})^{\oplus |\mathsf{B}|} \oplus H_k(\id \times \sigma \times \id)^{\oplus |\mathsf{A}-\mathsf{B}|}$ in a trivial fashion. 
   Indeed, the only elements that have a chance to get mapped to the kernel are of the form $x \times [\RP^3] \times {[1]}$, where $x \in H_{k-3}(T^3)$, because the restriction of $(\id_{T^3} \times \bm{\sigma}) \circ \iota_2$ and $(\id_{T^3} \times \sigma \times \id_{S^2}) \circ \iota_2$ to $T^3 \times S^2$ are the identity. 
   But $[\RP^3] \times {[1]}$ gets mapped to zero by $H_3(\bm{\sigma})$ or $H_3(\sigma \times \id_{S^2})$ for the target has trivial third integral homology.

   This observation has the following two consequences: 
    First, with slight abuse of notation, $H_k(\id \times \bm{\sigma})^{\oplus |\mathsf{B}|} \oplus H_k(\id \times \sigma \!\times \! \id)^{\oplus |\mathsf{A}-\mathsf{B}|}$ maps the image of $(H_k(\iota_1),H_k(\iota_2))$ to the image of $(H_k(\iota_1),H_k(\iota_2))$ in an isomorphic fashion.
    In particular, these two homomorphisms have the same kernel, so, equivalently, the boundary operators have the same image.
    Secondly, the kernel of  $H_k(\id \times \bm{\sigma})^{\oplus |\mathsf{B}|} \oplus H_k(\id \times \sigma \!\times \! \id)^{\oplus |\mathsf{A}-\mathsf{B}|}$ is mapped to $H_k(E_M)$ in an injective manner.
    A diagram chase implies that its image is precisely $\mathrm{ker}(H_k(q))$, and this observation implies the second part of the statement.

    The same line of reasoning applies to each fibre, so the fibre inclusion $i_x \colon M \hookrightarrow E_M$ (over $x \in S^2$) induces a map between two short exact sequences
    \begin{equation*}
        \xymatrix{\mathrm{ker}(H_k(\id \times \sigma))^{\oplus |\mathsf{B}|} \oplus \mathrm{ker}(H_k(\id \times \sigma))^{\oplus |\mathsf{A} - \mathsf{B}|}  \ar@{^{(}->}[r] \ar[d]^{H_k(i_x)^{\oplus |\mathsf{A} - \mathsf{B}|}}_{H_k(i_x)^{\oplus|\mathsf{B}|}} & H_k(M) \ar@{->>}[r]^{H_k(q_x)} \ar[d]^{H_k(i_x)} & H_k(N) \ar[d]^{H_k(\iota_x)}  \\
        \mathrm{ker}(H_k(\id \times \bm{\sigma}))^{\oplus |\mathsf{B}|} \oplus \mathrm{ker}(H_k(\id \times \sigma \times \id))^{\oplus |\mathsf{A} - \mathsf{B}|} \ar@{^{(}->}[r] & H_k(E_M) \ar@{->>}[r]^{H_k(q)}  & H_k(N \times S^2) .}
    \end{equation*}

    By the explicit description of the kernels, we deduce that $H_3(i_x)$ is an isomorphism between them. 
    The Künneth formula implies that the fibre inclusion $N \rightarrow N \times S^2$ also induces a monomorphism between the homology groups in degree $3$, so the first part of the statement follows from the Five Lemma.
\end{proof}

Any $M$-fibre bundle $E \rightarrow S^2$ can be represented by a continuous map $f_E \colon S^2 \rightarrow B\Diff(M)_0$ that is unique up to homotopy. 
Since $B\Diff(M)_0$ is simply connected, the classifying map gives rise to a unique\footnote{Since $B\Diff(M)_0$ is simply connected, pointed and unpointed homotopy classes are the same.} element $[f_E] \in \pi_2(B\Diff(M)_0)$. 

Let $\mathcal{S} \subseteq \pi_2(B\Diff(M)_0)$ be the subset whose elements classify those bundles $E \rightarrow S^2$ whose fibre inclusion induces a monomorphism between the third integral homology groups.

\begin{prop}\label{prop - Cohomology Invariant}
    The set $\mathcal{S}$ is a subgroup of $\pi_2(B\Diff(M)_0)$.
    Moreover, if $M$ is a manifold whose first Pontryagin class gives rise to a homomorphism $p_1(M) \cap \placeholder \colon H_4(M;\Z) \rightarrow m\Z$ for some $m \in \N$, then the map that assigns to an $M$-fibre bundle $E \rightarrow S^2$ the homomorphism $p_1(T^\vertical E) \cap \placeholder \colon H_4(E;\Z) \rightarrow \Z_m$ gives rise to a homomorphism
    $\mathcal{S} \rightarrow \mathrm{Hom}(H_2(M,\Z),\Z_m).$
\end{prop}
\begin{proof}
    Pick a bundle $E \rightarrow S^2$ whose classifying map is contained in $\mathcal{S}$ and consider the associated (homological) Leray-Serre spectral sequence
    \begin{equation*}
        E^2_{p,q} = H_p(S^2;H_q(M;\Z)) \Longrightarrow H_{p+q}(E;\Z).
    \end{equation*}
    As this spectral sequence converges to the homology of the total space, there is an ascending filtration $F_{p,q} := F_pH_{p+q}(E)$ of $H_{p+q}(E)$ that satisfies $F_{-1,p+q+1} = \{0\}$, $F_{p+q,0} = H_{p+q}(E)$, and $F_{p,q}/F_{p-1,q+1} \cong E^\infty_{p,q}$.
    We therefore have the following short exact sequence
    \begin{equation}\label{eq: extension problem}
        \xymatrix{ 0 \ar[r] & F_{0,4} \ar[r] \ar@{=}[d] & F_{2,2} \ar@{=}[d] \ar[r] & E^\infty_{2,2} \ar[r] \ar@{=}[d] & 0 \\
        & \mathrm{im}(H_4(\mathrm{incl})) & H_4(E;\Z), & \mathrm{ker} d_{2,2}^2 }
    \end{equation}
    where $\mathrm{incl} \colon M \rightarrow E$ is the fibre inclusion.
    We deduce that the fibre inclusion is a monomorphism between the third homology groups if and only if $d^2_{2,2} = 0$.

    Recall that the addition $[f_{E_1}] + [f_{E_2}]$ represents the bundle $E_1 + E_2 := c^\ast E_{1,2} \rightarrow S^2$ where $E_{1,2} \rightarrow S^2  \vee S^2$ is the unique bundle that restricts to $E_j$ on the $j$-th sphere, and $c \colon S^2 \rightarrow S^2 \vee S^2$ is the map that collapses the equator of $S^2$ to the basepoint.

    Let $A$ be an abelian group. The map $c$ induces the identity on the level of $H_0(\placeholder;A)$ and, under the isomorphism $H_2(S^2;A)^{\oplus 2} \xrightarrow{\cong} H_2(S^2 \vee S^2;A)$ that comes from the canonical inclusions $\iota_j \colon S^2 \rightarrow S^2 \vee S^2$, the homomorphism $H_2(c)$ corresponds to the diagonal map.

    With these observations, naturality of the Leray-Serre spectral sequence with respect to fibre maps implies that $d^{2,E_1 + E_2}_{2,2} = d_{2,2}^{2,E_1} + d_{2,2}^{2,E_2}$.
    Thus, we conclude that if two of the three bundles $E_1$, $E_2$ and $E_1 + E_2$ are classified by elements of $\mathcal{S}$, then so is the third.
    Since the constant map classifies the product $M \times S^2$, the zero element is contained in $\mathcal{S}$, hence it is a subgroup.

    \vspace{6pt}
  
    To prove the stament about the Pontryagin classes, we argue as follows.
    Bundles, whose classifying maps lie in $\mathcal{S}$, satisfy $d^2_{2,2} = 0$, so the short exact sequence (\ref{eq: extension problem}) becomes
    \begin{equation*}
         \xymatrix{ 0 \ar[r] & \mathrm{im}(H_4(\mathrm{incl})) \ar[r]  & H_4(E;\Z) \ar[r] & H_2\bigl(S^2;H_2(M;\Z)\bigr) \ar[r] & 0.}
    \end{equation*}
    By assumption, the restriction of $p_1(T^{\vertical} E)$ to $M$ is divisible by $m$, so the composition
    \begin{equation*}
        p_1(T^\vertical E) \cap \placeholder \colon H_4(E;\Z) \rightarrow \Z \rightarrow \Z_m
    \end{equation*}
    descends to a homomorphism $H_2\bigl(S^2;H_2(M;\Z)\bigr) \cong H_2(M;\Z) \rightarrow \Z_m$.

    To verify additivity, we derive for all $z \in H_2(S^2;H_2(M;\Z))$ the identity 
    \begin{align*}
       c_\ast \left( p_1\bigl(T^\vertical (E_1 + E_2)\bigr) \cap z \right) &= p_1(T^\vertical E_{1,2}) \cap c_\ast(z) \\
       &= p_1(T^\vertical E_{1,2}) \cap \bigl({\iota_1}_\ast(z) + {\iota_2}_\ast(z) \bigr) \\
       &= {\iota_1}_\ast \left( \iota_1^\ast p_1(T^\vertical E_{1,2}) \cap z \right) + {\iota_2}_\ast \left( \iota_2^\ast p_1(T^\vertical E_{1,2}) \cap z \right) \\
       &= {\iota_1}_\ast \bigl(p_1(T^\vertical E_1) \cap z\bigr) + {\iota_2}_\ast \bigl(p_1(T^\vertical E_2) \cap z\bigr).
     \end{align*}
     Since $H_2(c), H_2(\iota_1)$, and $H_2(\iota_2)$ are injective, additivity follows, and the claim is proved.
\end{proof}

With all these preparations, we are finally in the position to prove the following generalisation of Theorem \ref{Main Theorem - Example} and Theorem \ref{Main Theorem - pi_2-HMS}.

\begin{theorem}\label{Theorem - Lower Bound Elements Moduli}
    Let $M$ be a manifold obtained by a generalised Kummer constructions from a flat orbifolds satisfying Assumption \ref{Conditions - GroupsAction}. 
    Let $\mathsf{A}$ be the set of path components of the singularity set of this orbifold for which $F_a = \{1\}$ hold.
    Then $\pi_2(\hModuli{M})$, and hence $\pi_2(\Moduli{M})$ as well, contains at least $3^{|\mathsf{A}|}$ elements.
\end{theorem}
\begin{proof}
    Let $\mathsf{B} \subseteq \mathsf{A}$ be a subset.
    For each $b \in \mathsf{A}$, we have an inclusion $\iota_b$ as constructed in (\ref{eq - inclusion twisted core spheres}) or (\ref{eq - inclusion untwisted core spheres}).
    By Lemma \ref{Lemma - Integral along twisted core spheres} and Theorem \ref{Theorem - DivPontrClass} the first Pontryagin class of the vertical tangent bundle yields
    \begin{equation*}
        p_1(T^\vertical E_{M,\mathsf{B}}) \cap {\iota_b}_\ast ([\CP^1\times \CP^1]) = \langle p_1(T^\vertical E_{M,\mathsf{B}}) \, ; \, {\iota_b}_\ast[\CP^1\times \CP^1] \rangle = \begin{cases}
            - 8, & \text{ if } b\in B, \\
            0, & \text{ if } b \notin B.
        \end{cases}
    \end{equation*}
   Since we know by Proposition \ref{prop - Cohomology Invariant} that the assignment $E \mapsto p_1(T^\vertical E_M) \cap \placeholder$ gives rise to a homomorphism and since the images of $\CP^1 \times \CP^1$ under the embeddings $\iota_b$ form a free subgroup of rank $|\mathsf{A}|$ in $H_4(M;\Z)$, see Lemma \ref{Lemma - Fibre Inclusion Third Homology}, we deduce that the homomorphism
   \begin{equation*}
       \xymatrix{ \pi_2(B\Diff(M)_0) \supseteq \mathcal{S} \ar[r] & \mathrm{Hom}(H_2(M;\Z),\Z_3) \ar[rr]^-{\mathrm{restriction}} && \mathrm{Hom}(\Z^{|\mathsf{A}|};\Z_3)  }
   \end{equation*}
   is surjective.

   All the bundles $E_{M,\mathsf{B}}$ we used carry a fibrewise torsion-free $\GTwo$-structure on them, so we find at least $3^{|\mathsf{A}|}$ many elements inside $\pi_2(\hModuli{M})$.
   By Corollary \ref{Main Corollary - Coconnectivity}, the non-trivial elements we have found remain non-trivial in the moduli space, so $\pi_2(\Moduli{M})$  has at least $3^{|\mathsf{A}|}$ elements. 
\end{proof}

Theorem \ref{Theorem - Lower Bound Elements Moduli} can be applied to all examples in §3.2 of \cite{Joyce1996CompactG2II} with the exception of item 5 and 6 in the table of example 5 for their corresponding orbifolds do not satisfy condition (3) of our Conditions \ref{Conditions - GroupsAction}. 
In table (\ref{fig - Table Examples}) we list these examples with a lower bound of elements in the second homotopy group of their moduli space. 

\begin{figure}[h]
    \begin{tabular}{ |p{3cm}||p{1.2cm}|p{1.4cm}|p{1.2cm}|p{4cm}|  }
        \hline
      \multicolumn{5}{|c|}{Examples of Generalised Kummer Constructions from \cite{Joyce1996CompactG2II}} \\
        \hline
 Example Number & $\pi_1(M)$  & $b^3(M)$ & $|\mathsf{A}|$ & $|\pi_2(\Moduli{M})|$ \\
 \hline
 Example 03   & $\{1\}$   & $43$ & $12$ & $\geq 531\,441 $ \\
 \hline
 Example 04   & $\{1\}$    & $39$ -- $47$ & $8$ & $\geq 6\,561$ \\
 \hline
 Example 05 (ii)   & $\Z_2$    & $27$ -- $35$ & $4$ & $\geq 81$ \\
 \hline
 Example 05 (iii)   & $\Z_2$    & $21$ -- $29$ & $2$ & $\geq 9$ \\
 \hline
  Example 05 (iv)   & $\Z_2^2$    & $14$ -- $22$ & $1$ & $\geq 3$ \\
 \hline
  Example 06   & $\{1\}$    & $27$ -- $35$ & $2$ & $\geq 9$ \\
 \hline
  Example 07   & $\{1\}$    & $13$ & $2$ & $\geq 9$ \\
 \hline
  Example 08   & $\{1\}$    & $11$ & $2$ & $\geq 9$ \\
 \hline
  Example 09   & $\{1\}$    & $36$ & $10$ & $\geq 59\, 049  $ \\
 \hline
  Example 10   & $\{1\}$    & $21$ & $5$ & $\geq 243 $ \\
 \hline
  Example 11   & $\{1\}$    & $17$ & $4$ & $\geq 81 $ \\
 \hline
 Example 12   & $\Z_2$    & $11$ & $2$ & $\geq 9$ \\
 \hline
 Example 13   & $\{1\}$    & $10$ & $2$ & $\geq 9$ \\
 \hline
 Example 14   & $\{1\}$    & $10$ & $2$ & $\geq 9$ \\
 \hline
\end{tabular}
\caption{List of Joyce's Examples presented in §3.2 in \cite{Joyce1996CompactG2II} together with a lower bound of elements in $\pi_2$ of their moduli spaces. Here, $|\mathsf{A}|$ is the number of components with $F = \{1\}$.}
\label{fig - Table Examples}
\end{figure}

%% file: SliceTheorem.tex
\section{The Action of the Diffeomorphim Group}\label{Appendix - Slice Theorem}

In this section, we provide a slight strengthening of Joyce's slice theorem for torsion-free $\GTwo$-structures. 
We will prove that every torsion-free $\GTwo$-structure $\varphi_0$ is contained in a slice $(S,\mathrm{Stab}_{\varphi_0}(\Diff_0(M))$.
The reason is that the formulation as it appears in \cite[Theorem 10.4.1]{Joyce2000SpecialHolonomy} does not imply (SL3) in Definition \ref{Definition - Slice}. 
In addition, our formulation is closer to the original slice theorem due to Ebin \cite{Ebin1970Slice}, from which Joyce's slice theorem was inspired.

We assume throughout that $M^7$ is a closed, spin manifold, so that $\GTwoStr(M)$ is not empty.
Let $\Diff(M)$ be the topolgical group of all diffeomorphism of $M$. 
It is an open subset of $\mathcal{C}^\infty(M,M)$, the space of all smooth self-maps of $M$ endowed with the smooth Fr\'echet topology, see \cite[Theorem II.1.7]{hirsch1997differential}.  
Furthermore, $\Diff(M)$ into a smooth Fr\'echet Lie group, which is tame in the sense of \cite{Hamilton1982Inverse}, see especially Example I.4.4.5 and I.4.4.6 therein. 
Let $\Diff(M)_0$ be the path component of the identity in $\Diff(M)$.
This subgroup is open, and therefore closed. 
Furthermore, it inherits the structure of a tame Lie group from $\Diff(M)$.

The group $\Diff(M)$ acts smoothly from the right on $\GTwoStr(M)$ via pullback $\pb \colon \Diff(M) \curvearrowright \GTwoStr(M)$ and its differential at $(F,\varphi)$ is given by
\begin{align*}
    D_{(F,\varphi)}\mathrm{pb} \colon \Gamma(M,F^\ast TM) \times \Omega^3(M) &\rightarrow \Omega^3(M), & (X,\theta) \mapsto F^\ast \mathcal{L}_{DF^{-1}(X)}(\theta),
\end{align*}
where $\mathcal{L} \colon \Gamma(M,TM) \times \Omega^3(M) \rightarrow \Omega^3(M)$ is the Lie-derivative. 
\begin{lemma}
    The pullback action $\pb$ is proper, that is, preimages of compact sets under the map
    \begin{equation*}
        \Diff(M) \times \GTwoStr(M) \xrightarrow{(\pb,\id)} \GTwoStr(M) \times \GTwoStr(M)
    \end{equation*}
    are again compact.
\end{lemma}
\begin{proof}
    Consider the following commutative diagram:
    \begin{equation*}
        \xymatrix{\Diff(M) \times \GTwoStr(M) \ar[rr]^{(\pb,\mathrm{id})} \ar[d]_{\id \times}^f && \GTwoStr(M) \times \GTwoStr(M) \ar[d]_{f \times}^{f} \\
        \Diff(M) \times \mathrm{Riem}(M) \ar[rr]^{(\pb,\mathrm{id})} && \mathrm{Riem}(M) \times \Riem(M),}
    \end{equation*}
    where the vertical arrows are the maps $\varphi \mapsto g_{\varphi}$.
    
   Let $K \subseteq \GTwoStr(M) \times \GTwoStr(M)$ be a compact subset.
   We will show that every ultra filter on the preimage converges.

   Let $\mathcal{F}$ be an ultra filter on $(\pb,\id)^{-1}(K)$.
   As image filters of ultra filters are ultra filters again, $\pb(\mathcal{F})$ is an ultra filter on $K$, and so it converges.
   Thus, the image filer $\mathrm{pr}_2\bigl(\pb(\mathcal{F})\bigr) = \mathrm{pr}_2(\mathcal{F})$ converges to some element.

   The diagram commutes, so $\id \times f$ maps $(\pb,\id)^{-1}(K)$ into $((\pb,\id)^{-1}\bigl( (f\times f)(K) \bigr)$.
   But the target is compact, because the pullback action of $\Diff(M)$ on the space of Riemannian metrics is proper, see for example \cite{Corro2021Slice} or the original \cite{Ebin1970Slice}.
   Thus $\mathrm{pr}_1(\mathcal{F}) = \mathrm{pr}_1\bigl( (\id \times f)(\mathcal{F}) \bigr)$ converges, too.

   By the universal property of the product topology, $\mathcal{F}$ converges if and only if the image filters $\mathrm{pr}_1(F)$ and $\mathrm{pr}_2(\mathcal{F})$ converges, so $(\pb,\id)^{-1}(K)$ is compact.
\end{proof}

As $\GTwoStr(M)$ is an open subset of the Fr\'echet space $\Omega^3(M)$, its tangent bundle $T\GTwoStr(M)$ is trivial.
It can be endowed with a family of Riemannian metrics $\FinMetric{\placeholder}{\placeholder}_k$, for each $k \in \N_0$, defined by
\begin{equation*}
    \FinMetric{\theta_1}{\theta_2}_{k,\varphi} = \sum_{\alpha=0}^k \int_M \langle \nabla^\alpha \theta_1, \nabla^\alpha \theta_2 \rangle_{g^{(\alpha,1)}_\varphi} \diff \vol_{g_\varphi},
\end{equation*}
where $g^{(\alpha,1)}_\varphi$ is the Riemannian metric of $T^\vee M^{\otimes \alpha} \otimes TM$ induced by $g_\varphi$ and 
\begin{equation*}
    \nabla^\alpha \colon TM \xrightarrow{\nabla^{\varphi}} T^\vee M \otimes TM \rightarrow \dots \rightarrow T^\vee M^{\otimes \alpha} \otimes TM
\end{equation*}
is the composition of the corresponding Levi-Civita connections.
It is evident, that $\FinMetric{\placeholder}{\placeholder}_k$ is diffeomorphism invariant, that is
\begin{equation*}
    \FinMetric{F^\ast \theta_1}{F^\ast \theta_2}_{k,F^\ast \varphi} = \FinMetric{\theta_1}{\theta_2}_{k,\varphi}
\end{equation*}
for all $F \in \Diff(M)$, but see \cite[{Chap. 9}]{Bleecker1981Gauge} for quite general proof of this fact.
The induced norms $|||\placeholder|||_k$ of these metrics form a Finsler structure in the sense of \cite{Subramaniam1984Slice}, which means that
\begin{itemize}
    \item[(i)] for each $\varphi \in \GTwoStr(M)$, the norms $|||\placeholder|||_{k,\varphi}$ generate the Fr\'echet structure of $\GTwoStr(M)$, and
    \item[(ii)] for all $C>1$, all $\varphi \in \GTwoStr(M)$, and all $k \in \N_0$, there is an open neighbourhood $U$ of $\varphi$ such that, for all $\psi \in U$, the norms satisfy
    \begin{equation*}
        C^{-1}|||\placeholder|||_{k,\psi} \leq |||\placeholder|||_{k,\varphi} \leq C |||\placeholder|||_{k,\psi}.
    \end{equation*}
\end{itemize}
The first point follows from the fact that all $|||\placeholder|||_\varphi$ are equivalent to "classical" Sobolev-norms on manifolds obtained by patching Sobolev-norms on $\R^7$ together with a partition of unity (yielding a norm independent of $\varphi$) and the Sobolev embedding theorem, see \cite[Chapter 7]{bleecker2013index} for both.
The second point follows from the fact that the integrand of the inner product depends smoothly on $\varphi$.

\begin{lemma}
    Set $\pb_{\varphi_0} = \pb(\placeholder,\varphi_0)$ for each $\varphi_0 \in \GTwoStr(M)$.
    The pullback action $\pb$ is an elliptic action, which means:
    \begin{itemize}
        \item[(i)] The infinitesimal action $D_{\id}\pb_{\varphi_0} \colon T_{\id} \Diff(M) = \Gamma(TM) \rightarrow T_{\varphi_0}\Diff(M)\varphi_0 \subseteq \Omega^3(M)$ is an overdetermined elliptic differential operator\footnote{That means, its principal symbol is injective away from the zero section.}.
        \item[(ii)] The map $\GTwoStr(M) \rightarrow \Gamma(\Hom(J^1TM,\Lambda^3T^\vee M))$, in which $J^1TM$ denotes the bundle of one-jets, given by $\varphi \mapsto D_{\id} \pb_{\varphi}$ is also a differential operator.
    \end{itemize}
\end{lemma}
\begin{proof}
    Essentially by the definition of the Lie-derivative, we have $D_{\id} \pb_{\varphi}(X) = \mathcal{L}_X\varphi$.
    Hence, (ii) follows from Cartan's magic formula: $\mathcal{L}_X\varphi = \insertion{X} \diff \varphi + \diff \insertion{X} \varphi$.

    The commutator 
    \begin{align*}
        [D_{\id}\pb_{\varphi_0},f\cdot] = \mathcal{L}_{fX}\varphi_0 - f\mathcal{L}_X\varphi_0 = \diff f \wedge \insertion{X}\varphi_0 
    \end{align*}
    is $\mathcal{C}^\infty(M)$ linear, so $D_{\id}\pb_{\varphi_0}$ is a differential operator of order 1 with principal symbol
    \begin{equation*}
        \mathrm{symb}(D_{\id}\pb_{\varphi_0})(\xi)(\placeholder) = \xi \wedge \insertion{(\cdot)} \varphi_0 \colon TM \rightarrow \Lambda^3T^\vee M. 
    \end{equation*}
    
    Since the verification of injectivity if $\xi \neq 0 \in T^\vee_pM$ is a point-wise calculation, we may assume without loss of generality that $\varphi_0 = \diff t \wedge \omega^{\std} + \mathrm{Re}\, \Omega^{\std}$ is our model case on $\R^7 = \R \oplus \C^3$ and that $\xi = \xi_t \diff t$.
    If we decompose $X(p) = X_t(p) \partial_t + \left(X(p) - X_t(p)\partial_t\right) =: X_t(p)\partial_t + X(p)^\perp$, then $X(p)^\perp \in \C^3$ and a straight-forward computation yields that
    \begin{equation*}
        2\xi \wedge \insertion{X(p)}\varphi_0 =  \xi_t\diff t \wedge\Bigl( \underbrace{2X_t(p)\omega^{\std}}_{\in \Lambda^{1,1}\C^3} + \underbrace{\insertion{X(p)^\perp}\Omega^{\mathrm{std}}}_{\in \Lambda^{2,0}\C^3} +  \underbrace{\insertion{X(p)^\perp}\Bar{\Omega}^{\mathrm{std}}}_{\in \Lambda^{0,2}\C^3} \Bigr) \neq 0
    \end{equation*}
    whenever $\xi_t \neq 0$ and $X(p) \neq 0$.
\end{proof}

If $\varphi_0$ is a torsion-free $\GTwo$-structure, then the $L^2$-orthogonal complement of $T_{\varphi_0}\Diff(M)_0\varphi_0 = T_{\varphi_0}\Diff(M)\varphi_0$ inside $\Omega^3(M)$ is given by
\begin{equation*}
   \bigl(T_{\varphi_0}\Diff(M)_0\varphi_0\bigr)^\perp = \left\{ \theta \in \Omega^3(M) \, : \, \pi_7^2(\diff^{\ast_\varphi} \theta) = 0 \right\}, 
\end{equation*}
where $\pi_7^2 \colon \Omega^2(M) \rightarrow \Omega_7^2(M)$ is the projection induced by the projection $\pi_7^2 \colon \Lambda^2\R^7 \rightarrow \Lambda^2_7 := \{ \insertion{v}\varphi_{\std} \, : \, v \in \R^7 \}$ to the irreducible representation $\Lambda^2_7$ of $G_2$, see \cite[p.252]{Joyce2000SpecialHolonomy}.

Applying Subramaniam's slice theorem for elliptic actions, see \cite{Subramaniam1984Slice}, \cite{Subramniam1986Slice}, and \cite{Diez2019SliceTheorem} for a more general setup, to the pullback action $\mathsf{pb}\colon \Diff(M)_0 \curvearrowright \mathcal{G}_2(M)$ yields the following slight strengthening of Joyce's slice theorem\footnote{We can apply Subramaniam's slice theorem to the pullback action of the full diffeomorphism group to get a completely similar for the full diffeomorphism group.} \cite[Theorem 10.4.1]{Joyce2000SpecialHolonomy}:
\begin{theorem}[Slice Theorem]\label{Theorem - Slice Theorem}
    Let $\varphi_0 \in \tfGTwoStr(M)$ be a torsion-free $\GTwo$-structure on a closed manifold $M$
    and let $I_{\varphi_0}$ be the stabiliser of $\varphi_0$ inside $\Diff(M)_0$. Then inside
    $L_{\varphi_0} := \{ \varphi\in \GTwoStr(M) \, : \, \pi_7^2(\diff^{\ast_{\varphi_0}} \varphi)=0 \}$ exists a small open neighbourhood $S_{\varphi_0}$ such that the following statements hold true:
    \begin{itemize}
        \item[$\mathrm{(SL1)}$] $S_{\varphi_0}$ is $I_{\varphi_0}$-invariant: $S_{\varphi_0}\cdot I_{\varphi_0} \subseteq S_{\varphi_0}$
        \item[$\mathrm{(SL2)}$] If $F \in \Diff(M)_0$ satisfies $S_{\varphi_0} \cdot F \cap S_{\varphi_0} \neq \emptyset$, then $F \in I_{\varphi_0}$.
        \item[$\mathrm{(SL3)}$] $\Diff(M)_0 \rightarrow \Diff(M)_0/I_{\varphi_0}$ is a principal bundle and there is a local section $\chi\colon U \rightarrow \Diff(M)_0$ around the identity coset such that the map
        \begin{equation*}
            U \times S_{\varphi_0} \rightarrow \GTwoStr(M), \qquad ([F],\varphi) \mapsto \chi([F])^\ast \varphi 
        \end{equation*}
        is an open embedding.
    \end{itemize}
    In particular, the canonical inclusion induces an open embedding $S_{\varphi_0}/I_{\varphi_0} \rightarrow \GTwoStr(M)/\Diff(M)_0$.
\end{theorem}

Intersecting the 'slice' $S_{\varphi_0}$ from the previous theorem (or an appropriately chosen invariant, open subset of it) with $\tfGTwoStr(M)$, we deduce the following corollary for torsion-free $\GTwo$-structures.

\begin{cor}\label{Corollary - Slice Torsionfree}
    Let $\varphi_0 \in \tfGTwoStr(M)$ be a torsion-free $\GTwo$-structure on a closed manifold $M$
    and let $I_{\varphi_0}$ be the stabiliser of $\varphi_0$ inside $\Diff(M)_0$. Then inside
    $L_{\varphi_0}' := \tfGTwoStr(M) \cap \{ \varphi\in \GTwoStr(M) \, : \, \pi_7^2(\diff^{\ast_{\varphi_0}} \varphi)=0\}$ exists a small open neighbourhood $S_{\varphi_0}'$ such that the following statements hold true:
    \begin{itemize}
        \item[$\mathrm{(SL1)}$] $S_{\varphi_0}'$ is $I_{\varphi_0}$-invariant: $S_{\varphi_0}'\cdot I_{\varphi_0} \subseteq S_{\varphi_0}'$
        \item[$\mathrm{(SL2)}$] If $F \in \Diff(M)_0$ satisfies $S_{\varphi_0}' \cdot F \cap S_{\varphi_0}' \neq \emptyset$, then $F \in I_{\varphi_0}$.
        \item[$\mathrm{(SL3)}$] $\Diff(M)_0 \rightarrow \Diff(M)_0/I_{\varphi_0}$ is a principal bundle and there is a local section $\chi\colon U \rightarrow \Diff(M)_0$ around the identity coset such that the map
        \begin{equation*}
            U \times S_{\varphi_0}' \rightarrow \tfGTwoStr(M), \qquad ([F],\varphi) \mapsto \chi([F])^\ast \varphi 
        \end{equation*}
        is an open embedding.
    \end{itemize}
    In particular, the canonical inclusion induces an open embedding $S_{\varphi_0}'/I_{\varphi_0} \rightarrow \tfGTwoStr(M)/\Diff(M)_0$.
\end{cor}